\documentclass[twoside]{amsart}

\usepackage{amsmath,amssymb,amsfonts,amsthm,latexsym}
\usepackage{amscd,graphicx,color,enumerate}
\usepackage{hyperref}
\usepackage[margin=1in]{geometry}
\usepackage{mathrsfs}
\usepackage{tikz,url}
\numberwithin{equation}{section}
\usepackage{shuffle}
\usepackage{stmaryrd}
\usepackage{stackengine}
\stackMath
\usepackage{orcidlink}

\newtheorem{theorem}{Theorem}[section]
\newtheorem{lemma}[theorem]{Lemma}
\newtheorem{proposition}[theorem]{Proposition}
\newtheorem{corollary}[theorem]{Corollary}
\newtheorem{conjecture}[theorem]{Conjecture}

\theoremstyle{definition}
\newtheorem{definition}[theorem]{Definition}

\theoremstyle{remark}

\newcommand{\Z}{\mathbb{Z}}

\begin{document}

\title{On the numerology of trigonometric polynomials}

\author{Hans-Christian Herbig\,\orcidlink{0000-0003-2676-3340}}
\email{herbighc@im.ufrj.br}
\address{Departamento de Matem\'{a}tica Aplicada, Universidade Federal do Rio de Janeiro,
Av. Athos da Silveira Ramos 149, Centro de Tecnologia - Bloco C, CEP: 21941-909 - Rio de Janeiro, Brazil}
\author{Mateus de Jesus Gonçalves}
\email{mateusdejesusgoncalves@gmail.com}
\address{Departamento de Matem\'{a}tica Aplicada, Universidade Federal do Rio de Janeiro,
Av. Athos da Silveira Ramos 149, Centro de Tecnologia - Bloco C, CEP: 21941-909 - Rio de Janeiro, Brazil}


\begin{abstract} 
This note is about the observation that the various transition formulas between bases of trigonometric polynomials can be expressed in terms binomial coefficients. More specifically, we write the entries of the Chebyshev matrices $ T$ and $ U$ in terms of binomial coefficients. We remark that the inverses of certain submatrices of the Chebyshev matrix $ T$ correspond to classical power reduction formulas for trigonometric functions. Therefore the entries of $ T^{-1}$ are expressible in terms of binomial coefficients and powers of two. The inverses of certain submatrices of the Chebyshev matrix $ U$ appear in power reduction formulas involving the Catalan triangles. We could not spot them in the literature. As a corollary the Catalan triangles are interpreted as inverses of rather natural matrices of binomial coefficients. We explain the matrix inversions in terms of Riordan arrays.
We solve the integral $ \int _{0}^{2\pi }\cos^{2m}( x)\sin^{2n}( x) dx$ and deduce a hypergeometric identity from the Fourier expansion of $\cos^{2m}(x)$ and $ \sin^{2n}( x)$. As a corollary we prove that super Catalan numbers are integers. We emphasize that from the point of view of these base changes between the trigonometric polynomials it is more natural to work with $2\cos(x)$ and $2\sin(x)$ instead of  $\cos(x)$ and $\sin(x)$, since the transition matrices become invertible over the integers. We do a similar analysis for the spread polynomials. This  enables us to make a conjecture of Goh and Wildberger more precise.
In this exposition we give a mostly self-contained account of the matter, the prerequisites just being four semesters of calculus, linear algebra, elementary complex variables and the residue theorem.
 \end{abstract}
\maketitle


\tableofcontents


\section{Introduction}
\label{sec:Intro}

When teaching basic calculus and trigonometry the educator is faced with the challenge
that he or she cannot employ in a systematic way binomial coefficients, as Pascal's triangle is said to be too complicated for the modern student. On the other hand, since the advent of the Wilf-Zeilberger \cite{AequalsB} method binomial coefficients  are considered to be too shallow to be a subject of serious research. In the end, nobody appears to be responsible for this old subject, except the Online Encyclopaedia of Integer Sequences (OEIS) \cite{OEIS}. This invaluable resource contains no proofs however, but merely references. To find proofs of classical formulas in the literature can be a tedious task.

In this article we take the opportunity to discuss a number of trigonometric formulas
 that arise when doing base changes between the various bases of trigonometric polynomials. The formulas are all classical, even though the ones involving Catalan triangles we were not able to spot in the literature. We do not know of a reference where they are discussed in context and proofs are provided. Moreover, we try to exhibit the appearance of binomial coefficients as succinctly as possible and deduce the relevant simple hypergeometric identities behind the inversions. Among other things, we point out the close relation between the coefficients of the Chebyshev 
 polynomials and higher dimensional pyramidal numbers, a link we stumbled into while contemplating on the OEIS. 
 This is closely related to the observation that the base change formulas simplify when working with $2\cos(x)$ and $2\sin(x)$ instead of  $\cos(x)$ and $\sin(x)$ (but remember to keep the $1$ !).
We do a similar analysis for Norman Wildberger's spread polynomials that play a central role in his rational trigonometry \cite{Wildberger}. We express them in terms of the even-dimensional pyramidal numbers. The most elegant point of view on our matrix inversions is provided by the language of Riordan arrays \cite{ShapRiordan}, relating generating functions and matrix multiplication. Since we got rid of the annoying factors of four, we cleared the view to make a conjecture of Goh and Wildberger on the factorization of spread polynomials more precise. 

The paper is organized as follows. In Section \ref{sec:TU} we review the definitions of the Chebyshev polynomials. In Section \ref{sec:trigpol} we recall the notion of a trigonometric polynomial. In Section \ref{sec:pyr} we discover the higher pyramidal numbers and binomial coefficients in the coeficients of the Chebyshev polynomials. In Section \ref{sec:Mn} we propose an easy way to deduce quickly the first few Chebyshev polynomials $T_n$ on a sheet of paper. In Section \ref{sec:cat} (the only section in this paper where we refer to other sources) we recall the definition of the even and odd Catalan triangles. In Section \ref{sec:pr} we provide proofs of the well-known power reduction formulas and not-so-well-known variants thereof. In Section \ref{sec:inv} we relate the Chebyshev polynomials to the power reduction formulas via matrix inversion. In Section \ref{sec:riordan} we provide another proof of the matrix inversions using Riordan arrays. In Section \ref{sec:fc} we deduce further hypergeometric identities evaluating Fourier integrals and deduce the integrality of the super Catalan numbers. In Section \ref{sec:spread} we once again go through the laundry list of the paper for the case of the spread polynomials. In Section \ref{sec:GW} elaborate on a conjecture of Goh and Wildberger and provide some empirical evidence.
 
\vspace{2mm}

\noindent\textbf{Acknowledgements.} The paper arose from HCH teaching Calculus and Linear Algebra. The more HCH dug into the material he came to the realize that it should be known to Louis Shapiro in some form or another.
The paper is part of an undergraduate research project of MdJG. 
He is acknowledging financial support of PROFAEX.
We greatly benefited from discussions with Daniel Herden. HCH wants to thank A.-K. Gallagher for disposing \cite{Polya} on his bookshelf. HCH thanks his administration at UFRJ for forcing him to teach beginner's courses for a decade, turning him into an \emph{Elephant im Porzellanladen} (He really tried to teach them Eqn. \eqref{eq:weirdhyp}, but they do not want to listen!).

\vspace{2mm}
\noindent\textbf{Statements and Declarations.} The authors declare that they have no conflicts of interest. 

\section{Chebyshev polynomials}
\label{sec:TU}

For convenience of the reader we review here some basic facts about the Chebyshev polynomials (see, e.g., \cite[Exercise VI.2]{Polya}). The Chebyshev polynomials  of the first kind $(T_n(x))_{n\geq 0}$ and of the second kind $(U_n(x))_{n\geq 0}$ are defined by the recursions
\begin{align}
\nonumber   T_{0}( x) &=1,\ T_{1}( x) =x,\\
\label{eqn:Trec} T_{n}( x) &=2xT_{n-1}( x) -T_{n-2}( x) \ \mbox{ for $n\geq 2$},\\
\nonumber U_{0}( x) &=1,\ U_{1}( x) =2x,\\
\label{eqn:Urec} U_{n}( x) &=2xU_{n-1}( x) -U_{n-2}( x) \ \mbox{ for $n\geq 2$}.
\end{align}

\begin{proposition}
\begin{align}
\label{eqn:ChebGFT}\sum _{n=0}^{\infty } T_{n}( x) t^{n}&=\frac{1-tx}{1-2tx+t^{2}}\\
\label{eqn:ChebGFU}\sum _{n=0}^{\infty } U_{n}( x) t^{n}&=\frac{1}{1-2tx+t^{2}}
\end{align}
\end{proposition}
\begin{proof} From
     \begin{align*}
\left( 1-2tx+t^{2}\right)\sum _{n\geq 0} T_{n}( x) t^{n} &=\sum _{n\geq 0} T_{n}( x) t^{n} -\sum _{n\geq 0} T_{n}( x) 2t^{n+1}x +\sum _{n\geq 0} T_{n}( x) t^{n+2}\\
&=\sum _{n\geq 0} T_{n}( x) t^{n} -\sum _{m\geq 1} T_{m-1}( x)2t^{m}x +\sum _{m\geq 2} T_{m-2}( x) t^{m}\\
&=1+tx-2tx+\sum _{m\geq 2}\left( T_{m}( x) -2xT_{m-1}( x) + T_{m-2}( x)\right) t^{m}=1-tx
\end{align*}
we deduce Eqn. \eqref{eqn:ChebGFT}; the proof of \eqref{eqn:ChebGFU} is similar.
\end{proof}
\begin{theorem} For each integer $n\geq 0$ we have
\begin{align*}
    T_{n}(\cos( \theta)) &=\cos(n\theta) ,\qquad
    T_{n}(\sin( \theta)) =( -1)^{\lfloor n/2\rfloor }\begin{cases}
\cos( n\theta) & \mbox{if} \ n\ is\ \mbox{even}\\
\sin( n\theta) & \mbox{if} \ n\ is\ \mbox{odd}
\end{cases} ,\\
U_{n}(\cos( \theta)) &=\frac{\sin(( n+1) \theta)}{\sin(\theta)},\qquad
    U_{n}(\sin( \theta))=( -1)^{\lfloor n/2\rfloor }\begin{cases}
\frac{\cos(( n+1)\theta)}{\cos( \theta )} & n\ \text{even}\\
\frac{\sin(( n+1)\theta )}{\cos( \theta )} & n\ \text{odd}
\end{cases}.
\end{align*}
\end{theorem}
\begin{proof}
We give an inductive proof of $T_{m}(\cos( \theta)) =\cos(m\theta)$ and $\sin(\theta)U_{m}(\cos( \theta)) =\sin(m\theta)$. The start of the induction is obvious.  
Let us assume that these identities hold for $m=n-1$ and $m=n-2$.
Let us put for simplicity $U_{m}( x) =0=T_{m}( x)$ for $m<0$. From
\begin{align*}
\sum _{n\geq 0} xT_{n-1}( x) +\sum _{n\geq 0}\left( x^{2} -1\right) U_{n-2}( x) t^{n}=\sum _{n\geq 0} xT_{n}( x) t^{n+1} +\sum _{n\geq 0}\left( x^{2} -1\right) U_{n}( x) t^{n+2}\\
=xt\frac{1-tx}{1-2tx+t^{2}} +t^{2}\left( x^{2} -1\right)\frac{1}{1-2tx+t^{2}} =\frac{xt-t^{2}}{1-2tx+t^{2}} =\frac{-1+2tx-t^{2} +1-xt}{1-2tx+t^{2}}\\
=\frac{1-tx}{1-2tx+t^{2}} -1=\sum _{n\geq 1} T_{n}( x) t^{n},
\end{align*}
we see that $xT_{n-1}( x) +\left( x^{2} -1\right) U_{n-2}( x) =T_{n}( x)$ for $n\geq 1$. Similarly, we observe that
\begin{align*}
\sum _{n\geq 0} (T_{n}( x) t^{n} + xU_{n-1}( x)) t^{n} =\sum _{n\geq 0} T_{n}( x) t^{n} +\sum _{n\geq 0} xU_{n}( x) t^{n+1}
=\frac{1-tx}{1-2tx+t^{2}} +\frac{tx}{1-2tx+t^{2}} =\sum _{n\geq 0} U_{n}( x) t^{n},
\end{align*}
so that $T_{n}( x) +xU_{n-1}( x) =U_{n}( x)$ for $n\geq 0$. But, using the Addition Theorem for cosine,
\begin{align*}
\cos( n\theta ) =\cos( \theta )\cos(( n-1) \theta ) -\sin( \theta )\sin(( n-1) \theta ) =\cos( \theta ) T_{n-1}(\cos( \theta )) -\sin^{2}( \theta ) U_{n-2}(\cos( \theta )),
\end{align*}
the first relation proves $ T_{n}(\cos( \theta )) =\cos( n\theta )$.
Moreover, the second relation in combination with the Addition Theorem for sine
\begin{align*}
\sin(( n+1) \theta ) =\cos( \theta )\sin( n\theta ) +\sin( \theta )\cos( n\theta ) =\sin( \theta )(\cos( \theta ) U_{n-1}(\cos( \theta )) +T_{n}(\cos( \theta )))
\end{align*}
shows $ \sin( \theta ) U_{n}(\cos( \theta )) =\sin(( n+1) \theta )$.
\end{proof}

\section{Trigonometric polynomials}
\label{sec:trigpol}

A \textit{trigonometric polynomial} of degree $ \leq d$ in the variable $\theta$ is a univariate real function of the form
$ a_{0} +\sum _{n=1}^{d} a_{n}\cos( \theta ) +\sum _{n=1}^{d} b_{n}\sin( n\theta )$ with $ a_{0} ,a_{i} ,b_{i} \in \mathbb{R}$ for $ i=1,2,\dotsc ,d$.
The $ ( 2d+1)$-dimensional real vector space of trigonometric polynomials of degree $ \leq d$ \ will be denoted by $ \operatorname{Trig}_{\leq d}$. By $ 2\pi $-periodicity an $ f\in \operatorname{Trig}_{\leq d}$ is uniquely determined by its restriction to the interval $ [ -\pi ,\pi ]$. There is a inner product\footnote{In Fourier analysis it is more convenient to work with the inner product $2\langle f,g\rangle$.} 
\begin{align*}
\langle f,g\rangle =\frac{1}{2\pi }\int _{-\pi }^{\pi } f( \theta ) g( \theta ) d\theta 
\end{align*}
making $ \operatorname{Trig}_{\leq d}$ into an euclidean vector space. There is an orthogonal decomposition $ \operatorname{Trig}_{\leq d} =\operatorname{Trig}_{\leq d}^{0} \oplus \operatorname{Trig}_{\leq d}^{1}$ into even and odd trigonometric polynomials $ \operatorname{Trig}_{\leq d}^{j} :=\left\{f\in \operatorname{Trig}_{\leq d} |f( -\theta ) =( -1)^{j} f( \theta )\right\}$, $ j=0,1$. It is well-known that 
\begin{align*}
\begin{bmatrix}
1 & \sqrt{2}\cos( \theta ) & \sqrt{2}\cos( 2\theta ) &\sqrt{2} \cos( 3\theta ) & \dotsc  & \sqrt{2}\cos( d\theta )
\end{bmatrix}
\end{align*} 
forms an orthonormal basis for $ \operatorname{Trig}_{\leq d}^{0}$ and 
\begin{align*}
\begin{bmatrix}
\sqrt{2}\sin( \theta ) & \sqrt{2}\sin( 2\theta ) & \sqrt{2} \sin( 3\theta ) &  \dotsc  & \sqrt{2}\sin( d\theta )
\end{bmatrix}
\end{align*} 
forms an orthonormal basis for $ \operatorname{Trig}_{\leq d}^{1}$. In this paper we find it convenient to view a basis as a frame, i.e., a row of vectors with trivial kernel. We can understand the space of all trigonometric polynomials as the direct limit $ \operatorname{Trig} :=\cup _{d\geq 0}\operatorname{Trig}_{\leq d}$ with a decomposition $ \operatorname{Trig} =\operatorname{Trig}^{0} \oplus \operatorname{Trig}^{1}$, $ \operatorname{Trig}^{j} :=\cup _{d\geq 0}\operatorname{Trig}_{\leq d}^{j}$. It follows that
$ \begin{bmatrix}
1 &  \sqrt{2}\cos( \theta ) &  \sqrt{2}\cos( 2\theta ) &  \sqrt{2}\cos( 3\theta ) & \dotsc 
\end{bmatrix}$ forms an orthonormal basis for $ \operatorname{Trig}^{0}$ and that $ \begin{bmatrix}
 \sqrt{2}\sin( \theta ) &  \sqrt{2}\sin( 2\theta ) &  \sqrt{2} \sin( 3\theta ) & \dotsc 
\end{bmatrix}$ forms an orthonormal basis for $ \operatorname{Trig}^{1}$.

There are a lot of other (non-orthogonal) bases for $ \operatorname{Trig}^{j} ,\ j=0,1$. In this paper we are concerned, e.g., with the bases
\begin{align*}\begin{bmatrix}
\cos^{0}( \theta ) & \cos^{1}( \theta ) & \cos^{2}( \theta ) & \cos^{3}( \theta ) & \dotsc 
\end{bmatrix} ,\ \begin{bmatrix}
\frac{\sin^{1}( \theta )}{\sin( \theta )} & \frac{\sin^{2}( \theta )}{\sin( \theta )} & \frac{\sin^{3}( \theta )}{\sin( \theta )} & \frac{\sin^{4}( \theta )}{\sin( \theta )} & \dotsc 
\end{bmatrix} 
\end{align*}
for $ \operatorname{Trig}^{0}$ and 
\begin{align*}
\begin{bmatrix}
\sin( x) & \sin^{2}( x) & \sin^{3}( x) & \dotsc 
\end{bmatrix}
\end{align*} for $ \operatorname{Trig}^{1}$. To verify that these form bases we refer the reader to the classical literature (see, e.g., \cite{Polya}). Alternatively, the claim also follows from the calculations of Section \ref{sec:riordan}.

The Chebyshev polynomials $ T_{n}$ provide, for example, a base change between the bases
\begin{align*}
 \begin{bmatrix}
\cos^{0}( \theta ) & \cos^{1}( \theta ) & \cos^{2}( \theta ) & \cos^{3}( \theta ) & \dotsc 
\end{bmatrix}\mbox{ and } \begin{bmatrix}
1 & \cos( \theta ) & \cos( 2\theta ) & \cos( 3\theta ) & \dotsc 
\end{bmatrix},
\end{align*}
while the Chebyshev polynomials $ U_{n}$ \ provide a base change between the bases
\begin{align*}\begin{bmatrix}
\frac{\sin^{1}( \theta )}{\sin( \theta )} & \frac{\sin^{2}( \theta )}{\sin( \theta )} & \frac{\sin^{3}( \theta )}{\sin( \theta )} & \frac{\sin^{4}( \theta )}{\sin( \theta )} & \dotsc 
\end{bmatrix}\mbox{ and } \begin{bmatrix}
1 & \cos( \theta ) & \cos( 2\theta ) & \cos( 3\theta ) & \dotsc 
\end{bmatrix}.
\end{align*}
In the interest of keeping the entries of the base change matrices as simple as possible we advocate (see Section \ref{sec:inv}) for working with the basis $ \begin{bmatrix}
1 & 2\cos( \theta ) & 2\cos( 2\theta ) & 2\cos( 3\theta ) & \dotsc 
\end{bmatrix}$ instead of $ \begin{bmatrix}
1 & \cos( \theta ) & \cos( 2\theta ) & \cos( 3\theta ) & \dotsc 
\end{bmatrix}$ and with \ $ \begin{bmatrix}
2\sin( \theta ) & 2\sin^{2}( \theta ) & 2\sin^{3}( \theta ) & \dotsc 
\end{bmatrix}$ instead of \ $ \begin{bmatrix}
\sin( \theta ) & \sin^{2}( \theta ) & \sin^{3}( \theta ) & \dotsc 
\end{bmatrix}$. For the kabbalist the Chebyshev polynomials themselves appear to be a bit convoluted, and are probably defined this way for historical reasons. 

There is also a more modern variation of the subject. The so-called spread polynomials, introduced by Norman Wildberger (see \cite{Wildberger}), mediate a base change between $ \begin{bmatrix}
\sin^{2}( \theta ) & \sin^{4}( \theta ) & \sin^{6}( \theta ) & \dotsc 
\end{bmatrix}$ and $ \begin{bmatrix}
\sin^{2}( \theta ) & \sin^{2}( 2\theta ) & \sin^{2}( 3\theta ) & \dotsc 
\end{bmatrix}$. Again, from the point of view of numerology it seems more natural to work with $ \begin{bmatrix}
4\sin^{2}( \theta ) & 16\sin^{4}( \theta ) & 64\sin^{6}( \theta ) & \dotsc 
\end{bmatrix}$ and $ \begin{bmatrix}
4\sin^{2}( \theta ) & 4\sin^{2}( 2\theta ) & 4\sin^{2}( 3\theta ) & \dotsc 
\end{bmatrix}$ instead. The details of the base change are to be addressed in Section \ref{sec:spread}.

\section{Chebyshev matrix and higher dimensional pyramidal numbers}
\label{sec:pyr}
We define the infinite matrix $T=(T_{mn})_{m,n\geq 0}$
\begin{align}\label{eqn:Tee}
\footnotesize
T:= \left[\begin{array}{ccccccccccccc}
1 &  & -1 &   & 1 &   & -1 &  & 1 &   & -1 &  & \cdots \\
 & 1 &  & -3 &  & 5 &   &-7  & & 9 & & -11   & \\
 &   & 2 &  & -8 &  & 18 &  & -32 &  & 50 &  & \\
 &   &   & 4 &  & -20 &  & 56 &  &-120  &  & 220 & \\
 &   &  &   & 8 &  & -48 &  & 160 &  & -400 &  & \\
 &   &  &   &  & 16 &  & -112 &  & 432 &  & -1232 & \\
 &   &  &   &  &  & 32 &  & -256  &  & 1120 &  & \\
 &   &   &  &  &  &  & 64 &  & -567  &  &2816   & \\
 &  &    &   &  &  &  &  & 128 &  & -1280 &  & \\
 &  &    &   &  &  &  &  &  & 256 &  & -2816 & \\
 &  &    &    &  &  &  &  &  &  & 512 &  & \\
 &  &    &    &  &  &  &  &  &  &  & 1024 & \\
\cdots  &    &  &  &  &  &  &  &  &  &  &  & \cdots 
\end{array}\right],
\end{align}
so that $\begin{bmatrix}
    T_0(x)&T_1(x)&T_2(x)&\dots
\end{bmatrix}=\begin{bmatrix}
    1&x&x^2&\dots
\end{bmatrix}T$, i.e., $T_n(x)=\sum_{m\geq 0}T_{mn}x^m$. We take the liberty to not annotate zero entries.

Let us start our analysis by making an empirical observation.
For $ i\geq 1$ we record in a table the sequence $ \left( p_{j}^{[ i]}\right)_{j\geq 0}$ of nonzero entries of $ T$ \ $ i$th row multiplied by $ ( -1)^{j} 2^{-i+1}$. The $ 0$th row does not yet fit into the pattern.
\begin{align*}
\begin{array}{ c|c c c c c c c|c }
i\backslash j & 0 & 1 & 2 & 3 & 4 & 5 & \cdots & \text{Entry in OEIS}\\
\hline
1\text{th row} & 1 & 3 & 5 & 7 & 9 & 11 & \cdots & \text{A005408}\\
2\text{nd row} & 1 & 4 & 9 & 16 & 25 & 36 &  & \text{A000290}\\
3\text{rd row} & 1 & 5 & 14 & 30 & 55 & 91 &  & \text{A000330}\\
4\text{th row} & 1 & 6 & 20 & 50 & 105 & 196 &  & \text{A002415}\\
5\text{th row} & 1 & 7 & 27 & 77 & 182 & 378 &  & \text{A005585}\\
\cdots  &  &  &  &  &  &  & \cdots & 
\end{array}
\end{align*}
The number sequence $ \left( p_{j}^{[ i]}\right)_{j\geq 0}$ is usually referred to as the \emph{$i$-dimensional pyramidal numbers} (we are sloppy here about the index shifts). Its generating function
$ p^{[ i]}( t) =\sum _{j} p_{j}^{[ i]} t^{j}$ is simply
\begin{align*}
    p^{[ i]}( t)=\frac{1+t}{( 1-t)^{i+1}} =\frac{2-( 1-t)}{( 1-t)^{i+1}} =\frac{2}{( 1-t)^{i+1}} -\frac{1}{( 1-t)^{i}} =\sum _{j\geq 0}\left( 2\binom{i+j}{j} -\binom{i+j-1}{j}\right) t^{j}.
\end{align*}

The matrix 
\begin{align}\label{eqn:U}
\footnotesize
U=\left[\begin{array}{ccccccccccc}
1 &  & -1 &  & 1 &  & -1 &  & 1 &  & \\
 & 2 &  & -4 &  & 6 &  & -8 &  & 10 & \\
 &  & 4 &  & -12 &  &  24&  & -40 &  & \\
 & &  & 8 &  & -32  &  & 80 &  & -160 & \\
 &  &  &  & 16 &  & -80 &  & 240 &  & \\
 &  &  &  &  & 32 &  & -192 &  & 672 & \\
 &  &  &  &  &  & 64 &  & -448  &  & \\
 &  &  &  &  &  &  & 128 &  & -1024 & \\
 &  &  &  &  &  &  &  & 256 &  & \\
 &  &  &  &  &  &  & &  & 512 & \\
 &  &  &  &  &  &  &  &  &  & \cdots 
\end{array}\right],
\end{align}
on the other hand, is defined such that $U_n(x)=\sum_{m\geq 0}U_{mn}x^m$. For $ i\geq 0$ we record in a table the sequence of nonzero entries of the $i$th row of $U$ multiplied by $(-1)^{j} 2^{-i}$,
\begin{align*}
\begin{array}{ c|c c c c c c c }
i\backslash j & 0 & 1 & 2 & 3 & 4 & 5 & \cdots \\
\hline
0\text{th row} & 1 & 1 & 1 & 1 & 1 & 1 & \cdots \\
1\text{st row} & 1 & 2 & 3 & 4 & 5 & 6 &  \\
2\text{nd row} & 1 & 3 & 6  & 10 & 15 & 21 &  \\
3\text{rd row} & 1 & 4 & 10 & 20 & 35 & 56 &  \\
4\text{th row} & 1 & 5 & 15 & 35 & 70 & 126 & \\
\cdots   &  &  &  &  &  & &\cdots  
\end{array},
\end{align*}
which apparently turns out to be the Pascal matrix.

We phrase our observation as follows (this type of formulas is known for quite some time \cite{Clenshaw}).
\begin{theorem}\label{thm:ChebPyr}
    The Chebyshev polynomials are given by the formulas
    \begin{align}
   \label{eq:Tpoly}     T_{2n}( x) =( -1)^{n} +\sum _{j=1}^{n}( -1)^{n+j} 2^{2j-1} p_{n-j}^{[ 2j]} x^{2j} ,\ \ \ T_{2n+1}( x) =\sum _{j=0}^{n}( -1)^{n+j} 2^{2j} p_{n-j}^{[ 2j+1]} x^{2j+1} ,\\
\label{eq:Upoly} U_{2n}( x) =\sum _{j=0}^{n}( -1)^{n+j} 2^{2j}\binom{n+j}{2j} x^{2j} ,\ \ \ U_{2n+1}( x) =\sum _{j=0}^{n}( -1)^{n+j}2^{2j+1}\binom{n+j+1}{2j+1} x^{2j+1} ,
    \end{align}
    where $n\geq 0$ is an integer.
\end{theorem}
\begin{proof}
The first step is to cleanse Eqn. \eqref{eqn:Tee} from its powers of $2$. The trick is to define
\begin{align}\label{eqn:Pdef}
     P_{0}( z) &:=1, \\
\nonumber P_{n}( z) &:=2T_{n}( z/2) \ \ \mbox{for $n >0$}
\end{align}
to ensure that $P_{0}(1)=1$ and $P_{n}( 2\cos(\theta)) =2T_{n}(\cos(\theta))=2\cos( n\theta)$ for $n >0$. We easily deduce the bivariate generating function
\begin{align}\label{eqn:Pgf}
  \sum _{n\geq 0} P_{n}( z) t^{n} =\frac{2-zt}{1-zt+t^{2}} -1=\frac{1-t^2}{1-zt+t^{2}}. 
\end{align}
To cleanse the Taylor coefficients from the signs we make the substitution $t\mapsto t/\sqrt{-1} ,\ z\mapsto \sqrt{-1} z$
to obtain\footnote{This is \emph{almost} the generating function of the Lucas numbers $L(z)=\sum_{n\geq 0}L_nz^n=\frac{2-tz}{1-tz-t^2}$.}
\begin{align}
  \frac{1-t^{2}}{1-zt+t^{2}} \mapsto \frac{1+t^{2}}{1-zt-t^{2}}.
\end{align}
We extend our definition of the pyramidal numbers to include also $\sum _{j} p_{j}^{[0]}t^j=(1+t)/(1-t)=1+2t+2t^2+2t^3+\dots$.
Let $ K\subset \mathbb{C}^{\times }$ be compact and $ r=\operatorname{min}\left\{||t|-|t|^{-1} | \ \mid\ t\in K\right\}$. If $ |z|< r$ it follows that $ |\frac{tz}{1-t^{2}} |< 1\ $ for all $ t\in K$.
Then we have for $t\in K$, $ |z|< r$ 
\begin{align*}
\sum _{m,l\geq 0} p_{l}^{[ m]} z^{m} t^{m+2l} =\sum_{m\geq 0}\frac{1+t^{2}}{\left( 1-t^{2}\right)^{m+1}}( tz)^{m} =\sum_{m\geq 0}\frac{1+t^{2}}{1-t^{2}}\left(\frac{tz}{1-t^{2}}\right)^{m} =\frac{1+t^{2}}{1-t^{2}}\frac{1}{1-\frac{tz}{1-t^{2}}}
=\frac{1+t^{2}}{1-t^{2} -tz}.
\end{align*}
This holds everywhere outside of the poles, and obviously also at  $t=0$.
But now $\sum _{n} P_{n}( z) t^{n}=\frac{2-t^2}{1-zt+t^{2}}=\sum _{j} p_{j}^{[ i]} t^{2j}(-z)^i $. Similarly, we define
\begin{align}
V_{n}( z) &:=U_{n}( z/2) \ \ \mbox{for $n \geq 0$},
\end{align}
so that $V_{n}( 2\cos(\theta)) =U_{n}(\cos(\theta))=\sin( (n+1)\theta)/\sin(\theta)$ for $n\geq 0$. Hence the bivariate generating function is
\begin{align*}
  \sum _{n\geq 0} V_{n}( z) t^{n} =\frac{1}{1-zt+t^{2}}. 
\end{align*}
To get rid of the signs we make the substitution $t\mapsto t/\sqrt{-1} ,\ z\mapsto \sqrt{-1} z$
to obtain
$\frac{1}{1-zt+t^{2}} \mapsto \frac{1}{1-zt-t^{2}}$.
The relation
\begin{align*}
\sum _{m,l\geq 0}\binom{m+l}{l} z^{m} t^{m+2l} =\sum_{m\geq 0}\frac{( tz)^{m} }{\left( 1-t^{2}\right)^{m+1}} =\sum_{m\geq 0}\frac{1}{1-t^{2}}\left(\frac{tz}{1-t^{2}}\right)^{m} =\frac{1}{1-t^{2}}\frac{1}{1-\frac{tz}{1-t^{2}}} =\frac{1}{1-t^{2} -tz},
\end{align*}
holds for $t\in K$ and $|z|< r$, hence everywhere outside the poles.
\end{proof}

For practical matters, let us mention that for $i\geq 0$ the coefficients of $\sum _{j} p_{j}^{[ i]}t^j$
are the difference pattern of the coefficients of $\sum _{j} p_{j}^{[ i+1]}t^j$. This is because the generating functions differ by a factor $1-t$. We introduce 
 the infinite matrix $P=(P_{mn})_{n,m\geq 0}$ so that $P_n(x)=\sum_{m\geq 0}P_{mn}x^m$ (cf. Eqn. \eqref{eqn:Pgf}). 
\section{Mnemonics}
\label{sec:Mn}
If you are stranded on a lonesome island and want to recall quickly the Chebyshev polynomials $T_n(x)$ without entering some tedious recursions we have a suggestion. Start with the row of odd numbers and write on top of it its difference pattern. Then put on the bottom the row of the sequence whose differences are the odd numbers, (i.e., the square numbers). Proceed the same way with the resulting row and continue recursively. Introduce alternating signs for the diagonals and fill the gaps with zeros. The result is $ P$.
\begin{align*}
\begin{array}{ c c c c c c c c c c }
1 &  & 2 &  & 2 &  & 2 &  & 2&\cdots\\
 & 1 &  & 3 &  & 5 &  & 7 && \\
 &  & 1 &  &  4&  & 9 &  & 16&\cdots\\
\end{array}
\mapsto \begin{array}{ c c c c c c c c c c}
1 &  & 2 &  & 2 &  & 2 &  & 2&\cdots\\
 & 1 &  & 3 &  & 5 &  & 7 & &\\
 &  & 1 &  &  4&  & 9 &  & 16&\\
 &  &  & 1 &  & 5 &  & 14 & &\\
 &  &  &  & 1 &  & 6 &  & 20&\\
 &  &  &  &  & 1 &  & 7 & &\\
 &  &  &  &  &  & 1 &  &  8&\\
 &  &  &  &  &  &  & 1 & &\\
 &  &  &  &  &  &  &  & 1&\\
\cdots&&  &  & &  &  &  & &\cdots
\end{array} \\\mapsto P=\left[\begin{array}{ c c c c c c c c c c}
1 & 0 & -2 & 0 & 2 & 0 & -2 & 0 & 2&\cdots\\
0 & 1 & 0 & -3 & 0 & 5 & 0 & -7 & 0&\\
0 & 0 & 1 & 0 & -4 & 0 & 9 & 0 & -16&\\
0 & 0 & 0 & 1 & 0 & -5 & 0 & 14 & 0&\\
0 & 0 & 0 & 0 & 1 & 0 & -6 & 0 & 20&\\
0 & 0 & 0 & 0 & 0 & 1 & 0 & -7 & 0&\\
0 & 0 & 0 & 0 & 0 & 0 & 1 & 0 & -8&\\
0 & 0 & 0 & 0 & 0 & 0 & 0 & 1 & 0&\\
0 & 0 & 0 & 0 & 0 & 0 & 0 & 0 & 1&\\
\cdots&  &  &  &  &  &  & & &\cdots\\
\end{array}\right]
\end{align*}
To obtain $T$ turn all $2$s in the $0$th row into $1$s and for each $i\geq 1$ multiply the $i$th row with $2^{i}$.

There is another way to look at the coefficients of the Chebyshev polynomials that we would like to mention. We plan to discuss it in a more general context at another occasion. Namely, put
$$\begin{Bmatrix}
n\\
k
\end{Bmatrix} :=\begin{cases}
T_{00} /2 & \\
T_{2k\ n} & \text{if} \ n >0\ \text{is even},\\
T_{2k+1\ n} & \text{if} \ n\ \text{odd}.
\end{cases}$$
Then $\begin{Bmatrix}
n\\
k
\end{Bmatrix}$ satisfies the recurrence $\begin{Bmatrix}
n\\
k
\end{Bmatrix} =a\begin{Bmatrix}
n-1\\
k
\end{Bmatrix} +b\begin{Bmatrix}
n\\
k-1
\end{Bmatrix}$ with $a=2,\ b=-1$ and initial condition
$$\begin{Bmatrix}
n\\
0
\end{Bmatrix} =2^{n-1} , \ \begin{Bmatrix}
0\\
k
\end{Bmatrix} =\begin{cases}
1/2 & \text{if} \ k=0,\\
( -1)^{k} & \text{else}.
\end{cases} $$

\section{Catalan triangles and the Fuß-Catalan numbers}
\label{sec:cat}

The \emph{Catalan sequence} $(C_n)_{n\geq 0}$ can be defined by the Segner recursion
\begin{align*}
C_0&=1,\\
C_n&=\sum_{k=0}^{n-1}C_kC_{n-1-k}.
\end{align*}
In term of the generating function $C(x)=\sum_{n\geq 0}C_nx^n$ this can be written as $xC=C^2-1$. On other words, if you have forgotten $C_{n}$ but know $C_{0},C_1,\dots,C_{n-1}$, let your computer expand $(\sum_{k=0}^{n-1}C_kx^k)^2$. The coefficient in front of $x^{n-1}$ is $C_n$. By complementing the squares
one can determine the two roots of $C^2-xC-1$. This leads to the formula 
\begin{align}
    C(x)=\frac{1-\sqrt{1-4x}}{2x},
\end{align}
noting that only one root is a power series. From this one deduces $C_n={1\over n+1}{2n\choose n}={2n\choose n}-{2n-1\choose n}$. The literature on the Catalan sequence is extensive (see, e.g., \cite{Stanley,StanleyCat,concrete, Koshy}). A closely related function is the generating function of the \emph{central 
binomial coefficients}
\begin{align}
    B(x)=\sum_{n\geq 0} \binom{2n}{n} x^n=\frac{C(x)}{2-C(x)}=\frac{1}{\sqrt{1-4x}}.
\end{align}

The \emph{Fuß-Catalan numbers} are defined by
\begin{align}
F_{m}( p,r) :=\frac{1}{mp+r}\binom{mp+r}{m} =\frac{r}{m( p-1) +r}\binom{mp+r-1}{m} =\frac{r}{m}\binom{mp+r-1}{m-1}
\end{align}
for integers $m,p,r\geq 0$.
They satisfy the important convolution property $F_{m}( p,s+r) =\sum _{k=0}^{m} F_{k}( p,r) F_{m-k}( p,s)$; see \cite[Eqn. (5.36)]{concrete} or \cite{Riordan}. From this it follows that
\begin{align*}
    \sum_{m\geq 0}F_{m}( p,r)x^m=\left(\sum_{m\geq 0}F_{m}( p,1)x^m\right)^r,
\end{align*}
i.e., $\sum_{m\geq 0}F_{m}( p,r)x^m$ is the $r$-fold convolution of $\sum_{m\geq 0}F_{m}( p,1)x^m$. By inspection, $\sum_{m\geq 0}F_{m}(1,1)x^m=(1-x)^{-1}$ and  $\sum_{m\geq 0}F_{m}(2,1)x^m=C(x)$. The series $\sum_{m\geq 0}F_{m}(p,1)x^m$ 
is also called the \emph{generalized binomial series}.
Let us record the first $r$-fold convolutions of $C(x)$:
 \begin{align*}
C( x)  &=1+\ \ x+\ 2x^{2} +\ \ 5x^{3} + \ \ 14x^{4} + \ \ \ 42x^{5} +\ 132x^{6} + \ \ 429x^{7} +\dotsc ,\\
C^{2}( x) &=1+2x+\ 5x^{2} +\ 14x^{3} +\ \ 42x^{4} +\ \ 132x^{5} +\ 429x^{6} +\ 1430x^{7} +\dotsc ,\\
C^{3}( x) &=1+3x+9x^{2} +\ \ 28x^{3} +\ \ 90x^{4} +\ \ 297x^{5} +1001x^{6} +\ 3432x^{7} +\dotsc ,\\
C^{4}( x) &=1+4x+14x^{2} +\ 48x^{3} +\ 165x^{4} +\  572x^{5} +2002x^{6} +\ 7072x^{7} +\dotsc ,\\
C^{5}( x) &=1+5x+20x^{2} +\ 75x^{3} +\ 275x^{4} +1001x^{5} +3640x^{6} +13260x^{7} +\dotsc ,\\
C^{6}( x) &=1+6x+27x^{2} +110x^{3} +429x^{4} +1638x^{5} +6188x^{6} +23256x^{7} +\dotsc ,\\
C^{7}( x) &=1+7x+35x^{2} +154x^{3} +637x^{4} +2548x^{5} +9996x^{6} +38760x^{7} +\dotsc ,\\
C^{8}( x) &=1 + 8 x + 44 x^2 + 208 x^3 + 910 x^4 + 3808 x^5 + 15504 x^6 + 62016 x^7 + \dotsc 
\dotsc &
\end{align*}

The coefficients of the $2l$-fold convolutions $ C^{2l}( x)$ for $ l\geq 1$ have been arranged by Louis Shapiro \cite{Shapiro} into an infinite lower triangular matrix $ B^{\operatorname{even}}$ that is usually referred to as the Catalan triangle. We will call it here the \textit{even Catalan triangle} since we also need the \emph{odd Catalan triangle }$ B^{\operatorname{odd}}$. The latter is formed by arranging the $ ( 2l-1)$-fold convolutions $ C^{2l-1}( x)$ for $ l\geq 1$ into an infinite lower triangular matrix.
\begin{align}
B^{\operatorname{even}} =\begin{bmatrix}
1 &  &  &  &  & &\cdots \\
2 & 1 &  &  &  & &\\
5 & 4 & 1 &  &  & &\\
14 & 14 & 6 & 1 &  & &\\
42 & 48 & 27 & 8 & 1 & &\\
132 & 165 & 110 & 44 & 10 & 1&\\
\cdots  &  &  &  &  & & \cdots 
\end{bmatrix} ,\ \ \ \ \ B^{\operatorname{odd}} =\begin{bmatrix}
1 &  &  &  &  &  & \cdots \\
1 & 1 &  &  &  &  & \\
2 & 3 & 1 &  &  &  & \\
5 & 9 & 5 & 1 &  &  & \\
14 & 28 & 20 & 7 & 1 &  & \\
42 & 90 & 75 & 35 & 9 & 1 & \\
\cdots  &  &  &  &  &  & \cdots 
\end{bmatrix} .
\end{align}
The matrix entries are given by the formulas 
\begin{align}\label{eq:CatalanMatrices}
B_{ij}^{\operatorname{even}} =\frac{j}{i}\binom{2i}{i-j} ,\ \ B_{ij}^{\operatorname{odd}} =\frac{2j+1}{2i+1}\binom{2i+1}{i-j},
\end{align}
where $i,j\geq 1$.
One of the authors met $B_{ij}^{\operatorname{even}}$ already on another occasion \cite{Chi}. We encounter the Catalan triangles in the next section.

\section{Power reduction formulas}
\label{sec:pr}

\begin{theorem} \label{thm:powerred}
\begin{enumerate}
    \item $2^{2n-1}\cos^{2n}( \theta) =\frac{1}{2}\binom{2n}{n} +\sum _{k=1}^n\binom{2n}{n-k}\cos( 2k\theta)$,
    \item $2^{2n}\cos^{2n+1}( \theta) =\sum _{k=0}^{n}\binom{2n+1}{n-k}\cos(( 2k+1) \theta)$,
    \item $2^{2n-1}\sin^{2n}( \theta) =\frac{1}{2}\binom{2n}{n} +\sum _{k=1}^n( -1)^{k}\binom{2n}{n-k}\cos( 2k\theta)$,
    \item $2^{2n}\sin^{2n+1}( \theta) =\sum _{k=0}^{n}( -1)^{k}\binom{2n+1}{n-k}\sin(( 2k+1) \theta)$.
\end{enumerate}
\end{theorem}

\begin{proof} The following proof of (1) we learnt from Daniel Herden.
\begin{align*}
    \sum _{k=1}^n\binom{2n}{n-k}\cos( 2k\theta)&=\sum _{k=1}^n\binom{2n}{n-k}\frac{e^{2k\sqrt{-1}\theta} +e^{-2k\sqrt{-1}\theta}}{2}=e^{-2n\sqrt{-1}\theta}\sum _{k=1}^n\binom{2n}{n-k}\frac{e^{2( n+k) \sqrt{-1}\theta} +e^{2( n-k) \sqrt{-1}\theta}}{2}\\
&=\frac{e^{-2n\sqrt{-1}\theta}}{2}\left(\sum _{k=1}^n\binom{2n}{n-k} e^{2( n-k) \sqrt{-1}\theta} +\sum _{k=1}^n\binom{2n}{n-k} e^{2( n+k) \sqrt{-1}\theta}\right)\\
&=\frac{e^{-2n\sqrt{-1}\theta}}{2}\left(\sum _{k=1}^n\binom{2n}{n-k} e^{2( n-k) \sqrt{-1}\theta} +\sum _{k=1}^n\binom{2n}{n+k} e^{2( n+k) \sqrt{-1}\theta}\right)\\
&=\frac{e^{-2n\sqrt{-1}\theta}}{2}\left(-\binom{2n}{n} e^{2n\sqrt{-1}\theta}+\sum _{j}\binom{2n}{j} e^{2j\sqrt{-1}\theta} \right)=\frac{e^{-2n\sqrt{-1}\theta}}{2}\left( 1+e^{2\sqrt{-1}\theta}\right)^{2n} -\frac{1}{2}\binom{2n}{n} \\
&=\frac{1}{2}\left( e^{\sqrt{-1}\theta} +e^{-\sqrt{-1}\theta}\right)^{2n} -\frac{1}{2}\binom{2n}{n}
=2^{2n-1}\cos^{2n}( \theta) -\frac{1}{2}\binom{2n}{n}.
\end{align*}
By shifting the argument we deduce (3):
\begin{align*}
    2^{2n-1}\sin^{2n}( \theta) &=2^{2n-1}\cos^{2n}\left( \theta-\frac{\pi }{2}\right) =\frac{1}{2}\binom{2n}{n} +\sum _{k=1}^n\binom{2n}{n-k}\cos\left( 2k\left( \theta-\frac{\pi }{2}\right)\right) \\
    &=\frac{1}{2}\binom{2n}{n} +\sum _{k=1}^n( -1)^{k}\binom{2n}{n-k}\cos( 2k\theta).
\end{align*}
To show (2) along the lines of Herden's proof we calculate
\begin{align*}
  \sum _{k=0}^{n}\binom{2n+1}{n-k}\cos(( 2k+1) \theta) &=\sum _{k=0}^{n}\binom{2n+1}{n-k}\frac{e^{( 2k+1) \sqrt{-1}\theta} +e^{-( 2k+1) \sqrt{-1}\theta}}{2}\\
   &=\frac{e^{-(2n+1)\sqrt{-1}\theta}}{2}\left(\sum _{k=0}^{n}\binom{2n+1}{n-k} e^{( 2( n+k) +2) \sqrt{-1}\theta} +\sum _{k=0}^{n}\binom{2n+1}{n-k} e^{2( n-k) \sqrt{-1}\theta}\right)\\
&=\frac{e^{-(2n+1)\sqrt{-1}\theta}}{2}\left(\sum _{j=n+1}^{2n+1}\binom{2n+1}{j} e^{2j\sqrt{-1}\theta} +\sum _{j=0}^{n}\binom{2n+1}{j} e^{2j\sqrt{-1}\theta}\right)\\
&=\frac{e^{-(2n+1)\sqrt{-1}\theta}}{2}\sum _{j=0}^{2n+1}\binom{2n+1}{j} e^{2j\sqrt{-1}\theta} =\frac{e^{( -2n-1) \sqrt{-1}\theta}}{2}\left( 1+e^{2\sqrt{-1}\theta}\right)^{2n+1}\\
&=\frac{1}{2}\left( e^{\sqrt{-1}\theta} +e^{-\sqrt{-1}\theta}\right)^{2n+1} =2^{2n}\cos^{2n+1}( \theta).
\end{align*}
By shifting the argument we deduce (4):
\begin{align*}
2^{2n}\sin^{2n+1}( \theta) &=2^{2n}\cos^{2n+1}\left( \theta-\frac{\pi }{2}\right) =\sum _{k=0}^{n}\binom{2n+1}{n-k}\cos\left(( 2k+1)\left( \theta-\frac{\pi }{2}\right)\right)\\
&=\sum _{k=0}^{n}\binom{2n+1}{n-k}\cos\left(( 2k+1) \theta-k\pi -\frac{\pi }{2}\right) =\sum _{k=0}^{n}( -1)^{k}\binom{2n+1}{n-k}\sin(( 2k+1) \theta).
\end{align*}
\end{proof}
The following banal calculation has a pretty matrix inversion as a consequence (see Eqn. \eqref{eq:PyrCat1}––\eqref{eq:PyrCat4}). 
\begin{corollary}\label{cor:Cat}
\begin{enumerate}
\item $ 2^{2n}\cos^{2n}( \theta) =\sum _{k=0}^{n}B_{nk}^{\operatorname{odd}} \ \frac{\sin(( 2k+1) \theta)}{\sin( \theta)}$,
    \item $2^{2n-1} \cos^{2n-1}( \theta ) =\sum _{k=1}^{n} B_{nk}^{\operatorname{even}}\ \frac{\sin( 2k\theta )}{\sin( \theta)}$.
\end{enumerate}
\end{corollary}
\begin{proof} We derive (1) of Theorem \ref{thm:powerred}:
\begin{align*}
    -2^{2n-1} 2n\sin( \theta)\cos^{2n-1}( \theta) &={d\over d\theta}\left( 2^{2n-1}\cos^{2n}( \theta)\right)=-\sum _{k=1}^n\binom{2n}{n-k} 2k\sin( 2k\theta)\\
   \Longrightarrow 2^{2n-1}\cos^{2n-1}( \theta)&=\sum _{k=1}^n\frac{k}{n}\binom{2n}{n-k}\frac{\sin( 2k\theta)}{\sin( \theta)},
\end{align*}
which shows (2).
Similarly, taking the derivative of (2) of Theorem \ref{thm:powerred} we obtain:
\begin{align*}
    -2^{2n}( 2n+1)\sin( \theta)\cos^{2n}( \theta)& ={d\over d\theta}\left( 2^{2n}\cos^{2n+1}( \theta)\right)=-\sum _{k=0}^{n}\binom{2n+1}{n-k}( 2k+1)\sin(( 2k+1) \theta)\\ \Longrightarrow 2^{2n}\cos^{2n}( \theta) &=\sum _{k=0}^{n}\frac{2k+1}{2n+1}\binom{2n+1}{n-k}\frac{\sin(( 2k+1) \theta)}{\sin( \theta)},
\end{align*}
proving (1).
\end{proof}

\section{Hypergeometric identities from matrix inversions}
\label{sec:inv}
In order to distill the relevant hypergeometric identities we choose to get rid of the superficial powers of $2$. We put 
\begin{align}
\kappa ( \theta) :=2\cos( \theta) ,\ \sigma ( \theta) :=2\sin( \theta) ,\ \nu _{m}( \theta) :=\frac{\sin(( m+1) \theta)}{\sin( \theta)}
\end{align}
and work with the following bases 
\begin{align*}
\begin{bmatrix}
1 & \kappa ( \theta) & \kappa ( 2\theta) & \kappa ( 3\theta) & \dotsc 
\end{bmatrix},\\
\begin{bmatrix}
1 & \nu _{1}( \theta) & \nu _{2}( \theta) & \nu _{3}( \theta) & \dotsc 
\end{bmatrix}
\end{align*}
for the even trigonometric polynomials and
\begin{align*}
\begin{bmatrix}
\sigma ( \theta) & \sigma ( 2\theta) & \sigma ( 3\theta) & \dotsc 
\end{bmatrix}
\end{align*}
for the odd trigonometric polynomials.

We are prepared to understand that $T_{2n}(\cos( \theta))=\cos(2n \theta)$ and 
Theorem \ref{thm:powerred} (1) can be interpreted as the following mutually inverse linear systems 
\begin{multline}
\label{eq:inv1}
\begin{bmatrix}
1& \kappa ( 2\theta)&\kappa ( 4\theta)&\kappa ( 6\theta)&\kappa ( 8\theta)&\cdots 
\end{bmatrix} \\
=\begin{bmatrix}
\kappa ^{0}( \theta)&
\kappa ^{2}( \theta)&
\kappa ^{4}( \theta)&
\kappa ^{6}( \theta)&
\kappa ^{8}( \theta)&
\cdots 
\end{bmatrix}\begin{bmatrix}
p_{0}^{[ 0]} &-p_{1}^{[ 0]}  & p_{2}^{[ 0]} &  -p_{3}^{[ 0]}& p_{4}^{[ 0]} & \cdots \\
 & p_{0}^{[ 2]} & -p_{1}^{[ 2]} &p_{2}^{[ 2]}   & -p_{3}^{[ 2]} & \\
 &  & p_{0}^{[ 4]} &-p_{1}^{[ 4]}  & p_{2}^{[ 4]} & \\
 & &  & p_{0}^{[ 6]} & -p_{1}^{[ 6]} & \\
 & &  &  & p_{0}^{[ 8]} & \\
\cdots  &  &  &  &  & \cdots 
\end{bmatrix},
\end{multline}
\begin{multline}
\label{eq:inv2}
\begin{bmatrix}
\kappa ^{0}( \theta)&
\kappa ^{2}( \theta)&
\kappa ^{4}( \theta)&
\kappa ^{6}( \theta)&
\kappa ^{8}( \theta)&
\cdots 
\end{bmatrix}\\=\begin{bmatrix}
1&
\kappa ( 2\theta)&
\kappa ( 4\theta)&
\kappa ( 6\theta)&
\kappa ( 8\theta)&
\cdots 
\end{bmatrix}\begin{bmatrix}
\binom{0}{0} & \binom{2}{1} & \binom{4}{2} & \binom{6}{3} & \binom{8}{4} & \cdots \\
 & \binom{2}{0} & \binom{4}{1} & \binom{6}{2} &\binom{8}{3}  & \\
 &  & \binom{4}{0} & \binom{6}{1} & \binom{8}{2} & \\
 &  &  & \binom{6}{0} &\binom{8}{1}  & \\
 &  &  &  & \binom{8}{0} & \\
\cdots  &  &  &  &  & \cdots 
\end{bmatrix}.
\end{multline}
The equivalence of the relation $T_{2n}(\sin( \theta))=(-1)^n\cos(2n\theta)$ and 
Theorem \ref{thm:powerred} (3) boils down to essentially the same matrix inversion.

Similary, $T_{2n+1}(\cos( \theta))=\cos((2n+1)\theta)$ can be compared with
Theorem \ref{thm:powerred} (2):
\begin{multline}
\label{eq:inv3}
\begin{bmatrix}
\kappa ( \theta)& \kappa ( 3\theta)&\kappa ( 5\theta)&\kappa ( 7\theta)&\kappa ( 9\theta)&\cdots 
\end{bmatrix} \\
=\begin{bmatrix}
\kappa ^{1}( \theta)&
\kappa ^{3}( \theta)&
\kappa ^{5}( \theta)&
\kappa ^{7}( \theta)&
\kappa ^{9}( \theta)&
\cdots 
\end{bmatrix}\begin{bmatrix}
p_{0}^{[ 1]} &-p_{1}^{[ 1]}  & p_{2}^{[ 1]} &  -p_{3}^{[ 1]}& p_{4}^{[ 1]} & \cdots \\
 & p_{0}^{[ 3]} & -p_{1}^{[ 3]} &p_{2}^{[ 3]}   & -p_{3}^{[ 3]} & \\
 &  & p_{0}^{[ 5]} &-p_{1}^{[ 5]}  & p_{2}^{[ 5]} & \\
 & &  & p_{0}^{[ 7]} & -p_{1}^{[ 7]} & \\
 & &  &  & p_{0}^{[ 9]} & \\
\cdots  &  &  &  &  & \cdots 
\end{bmatrix},
\end{multline}
\begin{multline}
\label{eq:inv4}
\begin{bmatrix}
\kappa ^{1}( \theta)&
\kappa ^{3}( \theta)&
\kappa ^{5}( \theta)&
\kappa ^{7}( \theta)&
\kappa ^{9}( \theta)&
\cdots 
\end{bmatrix}\\=\begin{bmatrix}
\kappa ( \theta)& \kappa ( 3\theta)&\kappa ( 5\theta)&\kappa ( 7\theta)&\kappa ( 9\theta)&\cdots 
\end{bmatrix}\begin{bmatrix}
\binom{1}{0} & \binom{3}{1} & \binom{5}{2} & \binom{7}{3} & \binom{9}{4} & \cdots \\
 & \binom{3}{0} & \binom{5}{1} & \binom{7}{2} &\binom{9}{3}  & \\
 &  & \binom{5}{0} & \binom{7}{1} & \binom{9}{2} & \\
 &  &  & \binom{7}{0} &\binom{9}{1}  & \\
 &  &  &  & \binom{9}{0} & \\
\cdots  &  &  &  &  & \cdots 
\end{bmatrix}.
\end{multline}
Essentially the same matrix inversion shows up when comparing $T_{2n+1}(\sin( \theta))=(-1)^n\sin((2n+1)\theta)$ and 
Theorem \ref{thm:powerred} (4).

Finally, $U_{2n}(\cos( \theta))=\nu_{2n}(\theta)$ is the inverse linear system to Corollary \ref{cor:Cat} (1) 
\begin{multline}\label{eq:PyrCat1}
    \begin{bmatrix}
1&
\kappa ^{2}( \theta)&
\kappa ^{4}( \theta)&
\kappa ^{6}( \theta)&
\kappa ^{8}( \theta)&
\cdots 
\end{bmatrix}\\=\begin{bmatrix}
1& \nu_2(\theta)&\nu_4(\theta)&\nu_6(\theta)&\nu_8(\theta)&\cdots 
\end{bmatrix}\begin{bmatrix}
1 &1 & 2 & 5 & 14& \cdots \\
 & 1 & 3 & 9 &28  & \\
 &  & 1 & 5 & 20& \\
 &  &  & 1&7  & \\
 &  &  &  & 1 & \\
\cdots  &  &  &  &  & \cdots 
\end{bmatrix},
\end{multline}
\begin{multline}\label{eq:PyrCat2}
 \begin{bmatrix}
1& \nu_2(\theta)&\nu_4(\theta)&\nu_6(\theta)&\nu_8(\theta)&\cdots 
\end{bmatrix}\\=\begin{bmatrix}
1&
\kappa ^{2}( \theta)&
\kappa ^{4}( \theta)&
\kappa ^{6}( \theta)&
\kappa ^{8}( \theta)&
\cdots 
\end{bmatrix}\begin{bmatrix}
\binom{0}{0} &- \binom{1}{1} & \binom{2}{2} & -\binom{3}{3} & \binom{4}{4} & \cdots \\
 & \binom{2}{0} & -\binom{3}{1} & \binom{4}{2} &-\binom{5}{3}  & \\
 &  & \binom{4}{0} & -\binom{5}{1} & \binom{6}{2} & \\
 &  &  & \binom{6}{0} &-\binom{7}{1}  & \\
 &  &  &  & \binom{8}{0} & \\
\cdots  &  &  &  &  & \cdots 
\end{bmatrix},
\end{multline}
and $U_{2n+1}(\cos( \theta))=\nu_{2n+1}(\theta)$ is the inverse linear system to Corollary \ref{cor:Cat} (2) 
\begin{multline}\label{eq:PyrCat3}
    \begin{bmatrix}
\kappa &
\kappa ^{3}( \theta)&
\kappa ^{5}( \theta)&
\kappa ^{7}( \theta)&
\kappa ^{9}( \theta)&
\cdots 
\end{bmatrix}\\=\begin{bmatrix}
 \nu_1& \nu_3(\theta)&\nu_5(\theta)&\nu_7(\theta)&\nu_9(\theta)&\cdots 
\end{bmatrix}\begin{bmatrix}
1 & 2 & 5 & 14& 42&\cdots \\
 & 1 & 4 & 14 &48  & \\
 &  & 1 & 6 & 27& \\
 &  &  & 1&8  & \\
 &  &  &  & 1 & \\
\cdots  &  &  &  &  & \cdots 
\end{bmatrix},
\end{multline}
\begin{multline}\label{eq:PyrCat4}
 \begin{bmatrix}
\nu_1(\theta)& \nu_3(\theta)&\nu_5(\theta)&\nu_7(\theta)&\nu_9(\theta)&\cdots 
\end{bmatrix}\\=\begin{bmatrix}
\kappa ( \theta)&
\kappa ^{3}( \theta)&
\kappa ^{5}( \theta)&
\kappa ^{7}( \theta)&
\kappa ^{9}( \theta)&
\cdots 
\end{bmatrix}
\begin{bmatrix}
\binom{1}{0} &- \binom{2}{1} & \binom{3}{2} & -\binom{4}{3} & \binom{5}{4} & \cdots \\
 & \binom{3}{0} & -\binom{4}{1} & \binom{5}{2} &-\binom{6}{3}  & \\
 &  & \binom{5}{0} & -\binom{6}{1} & \binom{7}{2} & \\
 &  &  & \binom{7}{0} &-\binom{8}{1}  & \\
 &  &  &  & \binom{9}{0} & \\
\cdots  &  &  &  &  & \cdots 
\end{bmatrix}.
\end{multline}
The transposed matrices of Eqn. \eqref{eq:PyrCat2} and \eqref{eq:PyrCat4} are the odd and the even Catalan triangles $B^{\operatorname{odd}}$ and $B^{\operatorname{even}}$.
The latter two matrix inversions were already observed by Louis Shapiro in \cite{ShapiroOpen}, and we gave here a trigonometric proof.  In fact, all the matrices of this section are transposes of Riordan arrays. This can be used to explain all the inversions of this section; we do this in the next Section.

\section{Interpretation in terms of the Riordan group}
\label{sec:riordan}

The Chebyshev matrix $ T$ itself cannot be interpreted as the transpose of a Riordan array, but the matrizes $ P,B^{\operatorname{even}}$ and $ B^{\operatorname{odd}}$ can. This means that the matrix inversions of the previous section can be understood as inversions in the so-called Riordan group. The calculations are based on Lagrange inversion and turn out to be surprisingly simple. Some of the statements of this section can be found in some form in the PhD thesis of Aoife Hennessy \cite{Aoife}.

Let us review some basic ideas of the theory of Riordan arrays. The proofs of the facts reviewed here are actually quite accessible; we refer the reader to the monograph \cite{ShapRiordan}. We write $ \operatorname{\mathbb{R}} \llbracket x\rrbracket $ for the formal power series with real coefficients. By $ \operatorname{\mathbb{R}} \llbracket x\rrbracket ^{\times }$ we mean the multiplicatively invertible formal power series.

The composition map $ \operatorname{\operatorname{\mathbb{R}}}\llbracket x\rrbracket \times x\operatorname{\operatorname{\mathbb{R}}} \llbracket x\rrbracket \rightarrow \operatorname{\operatorname{\mathbb{R}}} \llbracket x\rrbracket $ \ that sends $ g( x) =\sum _{n\geq 0} g_{n} x^{n} \in \operatorname{\operatorname{\mathbb{R}}} \llbracket x\rrbracket $ and $ f( x) =\sum _{m\geq 1} f_{m} x^{m} \in x\operatorname{\mathbb{R}} \llbracket x\rrbracket $ to $ \sum _{n\geq 0} g_{n}\left(\sum _{m\geq 1} f_{m} x^{m}\right)^{n}$ is denoted by $ ( g,f) \mapsto g\circ f$. The space $ x\operatorname{\mathbb{R}} \llbracket x\rrbracket ^{\times }$ is comprised of compositionally invertible formal power series. I.e., if $ f\in x\operatorname{\mathbb{R}} \llbracket x\rrbracket ^{\times }$ then the map $ \operatorname{\operatorname{\mathbb{R}}} \llbracket x\rrbracket \rightarrow \operatorname{\operatorname{\mathbb{R}}} \llbracket x\rrbracket ,\ g\mapsto g\circ f$ is invertible. The inversion is provided by the Langrange inversion formula:  if $ f( x) =\frac{x}{\phi ( x)}$ then the coefficients of the \textit{compositional inverse }$ \overline{f}$ are given by the formula (see, e.g., \cite{Gessel})
\begin{align}
\label{eq:Lag}
\left[ y^{n}\right]\overline{f}\left( y^{n}\right) =\frac{1}{n}\left[ x^{n-1}\right]( \phi ( x))^{n}.
\end{align}

\begin{definition}

An infinite lower triangular matrix $ A=[ a_{nk}]_{n,k\geq 0} \in \operatorname{\mathbb{R}}^{\infty \times \infty }$ is called a \textit{Riordan array}
if there exist $ ( g,f) \in \operatorname{\mathbb{R}} \llbracket x\rrbracket \times x\operatorname{\mathbb{R}} \llbracket x\rrbracket^\times $ such that $ a_{nk} =\left[ x^{n}\right] g( x) f( x)^{k}$.
It is tradition to write, abusing notation, $ A=( g,f)$, so that $ ( g,f)$ can be interpreted as a matrix or as a tuple. If $ g\in \mathbb{R} \llbracket x\rrbracket ^{\times }$ then $ ( g,f)$ \ is called a \textit{proper} Riordan array.
\end{definition}

It turns out that $ \operatorname{\mathbb{R}} \llbracket x\rrbracket ^{\times } \times x\operatorname{\mathbb{R}} \llbracket x\rrbracket ^{\times }$ with the composition $ ( g_{1} ,f_{1}) *( g_{2} ,f_{2}) =( g_{1}( g_{2} \circ f_{1}) ,f_{2} \circ f_{1})$ forms a group, the so-called \textit{Riordan group}, with identity $(1,x)$ and inverse $ ( g,f)^{-1} =\left(\frac{1}{g\circ \overline{f}} ,\overline{f}\right)$.

The \textit{Fundamental Theorem of Riordan Arrays }links the Riordan group to matrix multiplication. Let us denote by $ \vec{h} =\begin{bmatrix}
h_{0} & h_{1} & h_{2} & \dotsc 
\end{bmatrix}^{T}$ the infinite column vector associated to the formal power series $ h( x) =\sum _{n\geq 0} h_{n} x^{n} \in \operatorname{\mathbb{R}} \llbracket x\rrbracket $. Then the theorem says that $ ( g,f)\vec{h} =\overrightarrow{g( h\circ f)}$. As a consequence one deduces the relation $ ( g_{1} ,f_{1})( g_{2} ,f_{2}) =\begin{bmatrix}
( g_{1} ,f_{1})\overrightarrow{g_{2}} & ( g_{1} ,f_{1})\overrightarrow{g_{2} f_{2}} & ( g_{1} ,f_{1})\overrightarrow{g_{2} f_{2}^{2}} & \dotsc 
\end{bmatrix} =( g_{1} ,f_{1}) *( g_{2} ,f_{2})$.

\begin{lemma}
\label{lem:RAmult}
Assume for $i=1,2$ the following inversions in the Riordan group $( g_{i} ,f)^{-1} =( G_{i} ,F)$.
Then $( g_{1} g_{2} ,f)^{-1} =( G_{1} G_{2} ,F)$ follows.
\end{lemma}

\begin{proof}
  The assumption means $( g_{i} ,f) *( G_{i} ,F) =( 1,x)$. It follows that
$$( g_{1} g_{2} ,f) *( G_{1} G_{2} ,F) =( g_{1} g_{2}( G_{1} G_{2} \circ F) ,f\circ F) =( g_{1}( G_{1} \circ F) g_{2}( G_{2} \circ F) ,f\circ F)=( 1,x).$$  
\end{proof}

The key observation we want to make now (it seems to be combinatorialist's folklore) is that
\begin{align*}
& B^{\operatorname{odd}} =\left( C,xC^{2}\right), \ \ \
B^{\operatorname{even}}=\left( C^{2} ,xC^{2}\right), \\
&\left( p^{[ 0]}( -x) ,x\frac{p^{[ 2]}( -x)}{p^{[ 0]}( -x)}\right) =\left(\frac{1-x}{1+x} ,\frac{x}{( 1+x)^{2}}\right),  \ \ \ \ \left( p^{[ 1]}( -x) ,x\frac{p^{[ 3]}( -x)}{p^{[ 1]}( -x)}\right) =\left(\frac{1-x}{( 1+x)^{2}} ,\frac{x}{( 1+x)^{2}}\right).
\end{align*}
are Riordan arrays. This permits to interprete Equations \eqref{eq:inv1}--\eqref{eq:PyrCat4} in Section \ref{sec:inv} as inversions in the Riordan group.

\begin{theorem} 
\label{thm:RABC}
Let $C$ and $B$ the generating functions of the Catalan numbers and of the central binomial coefficients, respectively. Then
\begin{enumerate}
    \item $ \left( C,xC^{2}\right)^{-1} =\left(\frac{1}{1+x} ,\frac{x}{( 1+x)^{2}}\right)$,

\item  $ \left( C^{2} ,xC^{2}\right)^{-1} =\left(\frac{1}{( 1+x)^{2}} ,\frac{x}{( 1+x)^{2}}\right)$,

\item  $ \left( B,xC^{2}\right)^{-1} =\left(\frac{1-x}{1+x} ,\frac{x}{( 1+x)^{2}}\right)$,

\item  $ \left( BC,xC^{2}\right)^{-1} =\left(\frac{1-x}{( 1+x)^{2}} ,\frac{x}{( 1+x)^{2}}\right)$.
\end{enumerate}
\end{theorem}

\begin{proof}

First we show $ \overline{xC^{2}}( x) =\frac{x}{( 1+x)^{2}}$ using formula \eqref{eq:Lag}. Namely, for $ n\geq 1$ we have
\begin{align*}
 \left[ y^{n}\right]\overline{\frac{x}{( 1+x)^{2}}}( y) =\frac{1}{n}\left[ x^{n-1}\right]( 1+x)^{2n} =\frac{1}{n}\left[ x^{n-1}\right]\sum _{j\geq 0}\binom{2n}{j} x^{j} =\frac{1}{n}\binom{2n}{n-1} =\left[ x^{n}\right] xC^{2}.
\end{align*}
To prove (1) it remains to check that $ C\circ \overline{xC^{2}} =1+x$. We do this by comparing $ C=C\circ \overline{xC^{2}} \circ \left( xC^{2}\right)$ \ with $ ( 1+x) \circ \left( xC^{2}\right) =\left( 1+xC^{2}\right) =1+C-1=C$, and remember that $ g\mapsto g\circ \left( xC^{2}\right)$ is invertible.

To prove (3) we observe that $ \left(\frac{1-xC^{2}}{1+xC^{2}}\right)^{-1} =\frac{C}{2-C} =B$ is the generating function for the central binomial coefficients.

Now (2) and (4) follow from (1) and (3) by Lemma \ref{lem:RAmult}.
\end{proof}

\begin{theorem}\label{thm:res}
Let $n,m$ be integers with $m\geq 0$. Then $\sum _{j\geq 0}\binom{2j+n}{j-m} x^{j} =BC^{n}( C-1)^{m}$.
\end{theorem}

\begin{proof} We generalize here a trick from \cite[Exercise 8.6]{Pap}.
Note that for a simple closed curve around $0\in \mathbb{C}$ we have for $N,M\in \mathbb{Z}$ that
$\binom{N}{M} =\frac{1}{2\pi \sqrt{-1}}\int _{\gamma }\frac{( 1+z)^{N}}{z^{M}}\frac{dz}{z}$. For $x\in \mathbb{C}$ with $|x|<\operatorname{dist(0,\gamma})$ we can swap summation and integral in the following calculation: 
\begin{align*}
\sum _{j\geq 0}\binom{2j+n}{j-m} x^{j} &=\sum _{j\geq 0}\frac{1}{2\pi \sqrt{-1}}\int _{\gamma }\frac{( 1+z)^{n+2j}}{z^{j-m}} x^{j}\frac{dz}{z}\\
&=\sum _{j\geq 0}\frac{1}{2\pi \sqrt{-1}}\int _{\gamma } z^{m}( 1+z)^{n}\left(\frac{( 1+z)^{2}}{z} x\right)^{j}\frac{dz}{z} =\frac{1}{2\pi \sqrt{-1}}\int _{\gamma } z^{m}( 1+z)^{n}\frac{1}{1-\frac{( 1+z)^{2}}{z} x}\frac{dz}{z}\\
&=\frac{1}{2\pi \sqrt{-1}}\frac{1}{x}\int _{\gamma } \frac{z^{m}( 1+z)^{n}}{\frac{z}{x} -( 1+z)^{2}} dz
=-\frac{1}{2\pi \sqrt{-1}}\frac{1}{x}\int _{\gamma } \frac{z^{m}( 1+z)^{n}}{( 1+z)^{2} -\frac{z}{x}} dz.
\end{align*}
We substitute $1+z=w,\ xw^{r} =( w-1)^{d}$ and write
\begin{align*}
\sum _{j\geq 0}\binom{n+2j}{k+j} x^{j} =-\frac{1}{2\pi \sqrt{-1}}\frac{1}{x}\int _{\gamma +1} \frac{w^{n}( w-1)^{m}}{w^{2} -\frac{w-1}{x}} dw.
\end{align*}
Factoring the quadratic polynomial in the denominator
\begin{align*}
w^{2} -\frac{w-1}{x} &=\left( w-\frac{1}{2x}\right)^{2} -\frac{1-4x}{4x^{2}} \\
&=\left( w-\frac{1-\sqrt{1-4x}}{2x}\right)\left( w-\frac{1+\sqrt{1-4x}}{2x}\right) =( w-C(x))\left( w-\frac{1}{x} +C(x)\right)
\end{align*}
and observing that $C(x)-\frac{1}{x} +C(x)=\frac{-1}{xB(x)}$ we can apply now the Residue Theorem and deduce the claim. In fact the second pole is outside $\gamma+1$ for $|x|$ small enough.
\end{proof}

\begin{corollary} $\left( B,xC^{2}\right) =\left[\binom{2n}{n-m}\right]_{n,m\geq 0} ,\ \left( BC,xC^{2}\right) =\left[\binom{2n+1}{n-m}\right]_{n,m\geq 0}$.
\end{corollary}

\begin{proof} According to the Theorem \ref{thm:res} we have
\begin{align*}
\sum _{n\geq 0}\binom{2n}{n-m} x^{n} =( C-1)^{m} B=x^mBC^{2m} ,\ \ \sum _{n\geq 0}\binom{2n+1}{n-m} x^{n} =C( C-1)^{m} B=x^mBC^{2m+1}.
\end{align*}
\end{proof}

\begin{corollary} For $|x|<1$ we have
    \begin{enumerate}
        \item $ \frac{1-x^{2}}{( 1+x)^{2} -4\cos^{2}( \theta ) x} =1+2\sum _{n\geq 1}\cos( 2n\theta ) x^{n}$,
 \item $ \frac{2( 1-x)\cos( \theta )}{( 1+x)^{2} -4x\cos^{2}( \theta )} =\sum _{n\geq 0}\cos(( 2n+1) \theta ) x^{n}$,
 \item $ \frac{1+x}{( 1+x)^{2} -4\cos^{2}( \theta )} =\sum _{n\geq 0}\frac{\sin(( 2n+1) \theta )}{\sin( \theta )} x^{n}$,
 \item $\frac{2\cos( \theta )}{( 1+x)^{2} -4x\cos^{2}( \theta )} =\sum _{n\geq 0}\frac{\sin( 2( n+1) \theta )}{\sin( \theta )} x^{n}$.
    \end{enumerate}
\end{corollary}

\begin{proof}
    We put
$\kappa _{\theta }^{\operatorname{even}}( x) :=\sum _{n\geq 0}\left( \kappa ^{2}( \theta ) x\right)^{n} =\frac{1}{1-\kappa ^{2}( \theta ) x}$
and note that the right hand side of Equation \eqref{eq:inv1} is
$$ \left(\frac{1-x}{1+x} ,\frac{x}{( 1+x)^{2}}\right)\overrightarrow{\kappa _{\theta }^{\operatorname{even}}} =\overrightarrow{\frac{1-x}{1+x}\frac{1}{1-\kappa ^{2}( \theta )\frac{x}{( 1+x)^{2}}}} =\overrightarrow{\frac{\left( 1-x^{2}\right)}{( 1+x)^{2} -x\kappa ^{2}( \theta ) }}.$$
This shows $ \frac{1-x^{2}}{( 1+x)^{2} -\kappa ^{2}( \theta ) x} =1+\sum _{n\geq 1} \kappa ( 2n\theta ) x^{n}$, which is equivalent to (1). Similarly, we put $\kappa _{\theta }^{\operatorname{odd}}( x) :=\sum _{n\geq 0} \kappa ( \theta )\left( \kappa ^{2}( \theta ) x\right)^{n} =\frac{\kappa ( \theta )}{1-\kappa ^{2}( \theta ) x}$
and express the right hand side of Equation \eqref{eq:inv3} as
$$ \left(\frac{1-x}{( 1+x)^{2}} ,\frac{x}{( 1+x)^{2}}\right)\overrightarrow{\kappa _{\theta }^{\operatorname{odd}}} =\overrightarrow{\frac{1-x}{( 1+x)^{2}}\frac{\kappa ( \theta )}{1-\kappa ^{2}( \theta )\frac{x}{( 1+x)^{2}}}} =\overrightarrow{\frac{( 1-x) \kappa ( \theta )}{( 1+x)^{2} -x\kappa ^{2}( \theta ) }}.$$
This shows $ \frac{( 1-x) \kappa ( \theta )}{( 1+x)^{2} -\kappa ^{2}( \theta ) x} =\sum _{n\geq 0} \kappa (( 2n+1) \theta ) x^{n}$, which is equivalent to (2). Writing the right hand side of Equation \ref{eq:PyrCat2} 
$$\left(\frac{1}{1+x} ,\frac{x}{( 1+x)^{2}}\right)\overrightarrow{\kappa _{\theta }^{\operatorname{even}}} =\overrightarrow{\frac{1}{1+x}\frac{1}{1-\kappa ^{2}( \theta )\frac{x}{( 1+x)^{2}}}} =\overrightarrow{\frac{1+x}{( 1+x)^{2} -x\kappa ^{2}( \theta )}}$$
we prove $ \frac{1+x}{( 1+x)^{2} -x\kappa ^{2}( \theta )} =\sum _{n\geq 0} \nu ( 2n\theta ) x^{n}$, and hence (3). Finally, writing the right hand side of Equation \eqref{eq:PyrCat4} as
$$ \left(\frac{1}{( 1+x)^{2}} ,\frac{x}{( 1+x)^{2}}\right)\overrightarrow{\kappa _{\theta }^{\operatorname{odd}}} =\overrightarrow{\frac{1}{( 1+x)^{2}}\frac{\kappa ( \theta )}{1-\kappa ^{2}( \theta )\frac{x}{( 1+x)^{2}}}} =\overrightarrow{\frac{\kappa ( \theta )}{( 1+x)^{2} -x\kappa ^{2}( \theta )}}$$
we deduce $ \frac{\kappa ( \theta )}{( 1+x)^{2} -x\kappa ^{2}( \theta )} =\sum _{n\geq 0} \nu (( 2n+1) \theta ) x^{n}$, which is equivalent to (4). 
\end{proof}

\section{Hypergeometric identities from trigonometric intergrals}
\label{sec:fc}

This section is not essential for the logic of the paper, but gives an example how the matrix inversions can be used. Our first aim is to deduce further hypergeometric identities by inspecting the trigonometric integral $\int _{0}^{2\pi }\cos^{n}( \theta )\sin^{m}( \theta ) d\theta $.
This could be evaluated using the beta function, but we prefer here an more elementary treatment. When $ n\geq 1,\ m\geq 0$ we perform the partial integration
\begin{align*}
\ I_{n,m} =\int _{0}^{2\pi }\cos^{n}( \theta )\sin^{m}( \theta ) \ d\theta 
=\left[\frac{\cos^{n-1}( \theta )\sin^{m+1}( \theta )}{m+1}\right]_{0}^{2\pi } \ +\int _{0}^{2\pi }( n-1) \ \sin( \theta ) \ \cos^{n-2}( \theta )\frac{\ \sin^{m+1}( \theta )}{m+1} \ d\theta.
\end{align*}
We deduce $ ( m+1) I_{n,m} =( n-1) \ I_{n-2,m+2}$. From the substitution  $ s=\theta -\pi /2$ we find 
\begin{align*}
I_{n,m} =\int _{0}^{2\pi }\cos^{n}( \theta )\sin^{m}( \theta ) d\theta =\int _{-\pi /2}^{3\pi /2}\left( -\sin( s)^{n}\right)\cos^{m}( s) ds=( -1)^{n} I_{m,n}. \end{align*}
Evidently, \ $ I_{n,1} =\int _{0}^{2\pi }\cos^{n}( \theta )\sin( \theta ) \ d\theta =\left[ -\frac{\cos^{n+1}( t)}{n+1}\right]_{0}^{2\pi } =0$. Moreover, from
the power reduction formulas (see Theorem \ref{thm:powerred}) we obtain
\begin{align}\label{eq:evenbasecase}
I_{n,0} =\begin{cases}
0 & \text{if} \ n\ \text{odd}\\
2\pi \ 2^{-n}\binom{n}{n/2} & \text{if} \ n\ \text{even}
\end{cases} .
\end{align}
Hence $I_{n,m} \neq 0$ if and only if $ n$ are $ m$ even. By induction we obtain for $ k\geq l\geq 0$:
\begin{align*}
 I_{2k,2l} =\left(\prod _{m=1}^{l} \ \ \frac{( 2m+1)}{( 2k+2m-1)}\right) \ I_{2( k+l) ,0}.
 \end{align*}
But $ \prod _{m=1}^{l}( 2k+2m-1) =\frac{( 2k+2l-1) !!}{( 2k-1) !!} \ $and hence $ $$ I_{2k,2l} =\frac{( 2k-1) !!}{( 2k+2l-1) !!\ \ }( 2l-1) !!\ I_{2( k+l) ,0}$. From
\begin{align*}
\frac{( 2k-1) !!}{( 2k+2l-1) !!}( 2l-1) !!=\frac{2^{k+l-1}( k+l-1) !}{{(2k+2l-1)!}}
\frac{(2k-1)!}{2^{k-1}( k-1) !} \ \frac{( 2l-1) !}{2^{l-1}( l-1) !}\\
=2\frac{\binom{k+l-1}{k-1} l!}{{\binom{2k+2l-1}{2k-1}}{(2l)!}}\frac{( 2l-1) !}{( l-1) !} =\frac{\binom{k+l-1}{k-1}}{\binom{2k+2l-1}{2k-1}} =\frac{\binom{k+l-1}{l}}{\binom{2k+2l-1}{2l}}
\end{align*}
Remembering $ I_{2k,2l} =I_{2l,2k}$, we finally deduce that for $ k,l\geq 0$
\begin{align*}
I_{2k,2l} =\frac{\binom{k+l-1}{l}}{\binom{2k+2l-1}{2l}} I_{2( k+l) ,0}.
\end{align*}
Combining this with \eqref{eq:evenbasecase} we have proven the following.

\begin{theorem}
\label{thm:trigint}
For integers $ k,l\geq 0$ we have
$ \int _{0}^{2\pi }\cos^{2k}( \theta )\sin^{2l}( \theta ) \ d\theta =2\pi \ 2^{-2( k+l)}\frac{\binom{k+l-1}{l}\binom{2( k+l)}{k+l}}{\binom{2( k+l) -1}{2l}}$.
\end{theorem}
We evaluate
\begin{align*}
2\pi =\int _{0}^{2\pi }\left(\cos^{2}( \theta ) \ +\sin^{2}( \theta )\right)^{m} \ d\theta =\sum _{{\substack{
k,l\geq 0\\
k+l=m
}}}\binom{k+l}{l}\int _{0}^{2\pi }\cos^{2k}( \theta ) \ \ \sin^{2l}( \theta ) \ d\theta \\
=2\pi \sum_{\substack{
k,l\geq 0\\
k+l=m
}}
\binom{k+l}{l} \ 2^{-2( k+l)}\frac{\binom{k+l-1}{l}\binom{2( k+l)}{k+l}}{\binom{2k+2l-1}{2l}}
\end{align*}
and find as a corollary 
\begin{align}
\label{eq:weirdhyp}
\frac{2^{2m}}{\binom{2m}{m}} =\sum _{l=0}^{m}\frac{\binom{m-1}{l}\binom{m}{l}}{\binom{2m-1}{2l}}.
\end{align}
The rest of this section we use our base change of the previous section to show that the infinite symmetric matrix $M=[M_{kl}]_{k,l\geq 0}$, 
\begin{align*}
M_{kl} :=\frac{1}{2\pi }\int _{-\pi}^{\pi } \kappa ^{2k}( \theta ) \sigma ^{2l}( \theta ) \ d\theta =\frac{\binom{k+l-1}{l}\binom{2( k+l)}{k+l}}{\binom{2( k+l) -1}{2l}}.
\end{align*}
has integer entries. On the way we deduce another hypergeometric identity (namely Eqn. \eqref{eq:LU} below). For the kabbalist\footnote{In the words of Fields medalist Edward Witten, $M$ stands for Magical, Mystery or Membrane.} the matrix $M$ looks actually a bit curious
\begin{align*}
    M=\begin{bmatrix}
1 & 2 & 6 & 20 & 70 & 252 & 924 & \cdots \\
2 & 2 & 4 & 10 & 28 & 84 & 264 & \\
6 & 4 & 6 & 12 & 28 & 72 & 198 & \\
20 & 10 & 12 & 20 & 40 & 90 & 220 & \\
70 & 28 & 28 & 40 & 70 & 140 & 308 & \\
252 & 84 & 72 & 90 & 140 & 252 & 504 & \\
924 & 264 & 198 & 220 & 308 & 504 & 924 & \\
\cdots  &  &  &  &  &  &  & \cdots 
\end{bmatrix}.
\end{align*}
The first row is comprised of the central binomial coefficients and the second row of doubled Catalan numbers. The lower triangular part of $M$ appears in the OEIS as entry \url{https://oeis.org/A182411}. Let us mention other empirical observations about the entries of $M$. It appears that $M_{kl}=\frac{( 2k) !( 2l) !}{k!l!( k+l) !}$ and the textbook \cite{Uspensky} asks in Exercise 4. on p.103 to show that this is an integer. Nowadays these numbers are called \emph{super Catalan numbers} (see \cite{superGessel, Borisov})\footnote{In \cite{superWild} the rational numbers in Theorem \ref{thm:trigint} have been named \emph{circular super Catalan numbers}.}. That our $M_{kl}$ are the super Catalan numbers can be easily verified
\begin{align*}
 \frac{( 2k+2l) !\ }{( k+l) !^{2}}\frac{\ ( 2l) !\ ( 2k-1) !\ !}{\ ( 2k+2l-1) !\ }\frac{\ ( k+l-1) !}{\ \ ( k-1) !l!} =\frac{2( k+l) \ ( 2l) !\ ( 2k) !\ }{( k+l) !( k+l) \ ( k-1) !l!2k} =\frac{\ ( 2l) !\ ( 2k) !\ }{( k+l) !\ ( k-1) !l!k}.  
\end{align*}
Moreover, according to computer algebra the matrix $M$ admits the neat LU factorization 
\begin{align}
\label{eq:LU}
M=L\operatorname{diag}(1,-2,2,-2,2,\dots)L^T,
\end{align}
where $L$ is the lower triangular matrix $\left[\binom{2i}{i-j}\right]_{i,j\geq 0}$. This entails that the upper left $n\times n$-block $M^{(n)}$ of $M$ has determinant $( -1)^{\lfloor n/2\rfloor } 2^{n-1}$ and that $M$ is integer. Once again, the number sequence $1,2,2,2,\dots$ pops up. 

It turns out that with our preparations we can easily find an explanation for Eqn. \eqref{eq:LU}. By the results of Section \ref{sec:inv} we have
\begin{align*}
\begin{bmatrix}
\kappa ^{0}( \theta)&
\kappa ^{2}( \theta)&
\kappa ^{4}( \theta)&
\kappa ^{6}( \theta)&
\cdots 
\end{bmatrix}&=\begin{bmatrix}
1& \kappa ( \theta)&\kappa ( 2\theta)&\kappa ( 4\theta)&\cdots 
\end{bmatrix}L^T,\\
\begin{bmatrix}
\sigma ^{0}( \theta)&
\sigma ^{2}( \theta)&
\sigma ^{4}( \theta)&
\sigma ^{6}( \theta)&
\cdots 
\end{bmatrix}&=\begin{bmatrix}
1& \kappa ( \theta)&\kappa ( 2\theta)&\kappa ( 4\theta)&\cdots 
\end{bmatrix}\operatorname{diag}(1,-1,1,-1,\dots)L^T
.
\end{align*}
Transposing and integrating we deduce
 \begin{align*}
M&=\frac{1}{2\pi }\int _{-\pi }^{\pi }\begin{bmatrix}
\kappa ^{0} (\theta )\\
\kappa ^{2} (\theta )\\
\kappa ^{4} (\theta )\\
\kappa ^{6} (\theta )\\
\cdots 
\end{bmatrix}\begin{bmatrix}
\sigma ^{0} (\theta ) & \sigma ^{2} (\theta ) & \sigma ^{4} (\theta ) & \sigma ^{6} (\theta ) & \cdots 
\end{bmatrix} d\theta \\
&=L\left(\frac{1}{2\pi }\int _{-\pi }^{\pi }\begin{bmatrix}
1\\
\kappa (2\theta )\\
\kappa (4\theta )\\
\kappa (6\theta )\\
\cdots 
\end{bmatrix}\begin{bmatrix}
1 & \kappa (2\theta ) & \kappa (4\theta ) & \kappa (6\theta ) & \cdots 
\end{bmatrix} d\theta \right)\operatorname{diag} (1,-1,1,-1,\dotsc )L^{T}\\
&=L\operatorname{\operatorname{diag} (1,2,2,2,\dotsc )diag} (1,-1,1,-1,\dotsc )L^{T}.
\end{align*}

\section{Spread polynomials}
\label{sec:spread}

The \emph{spread polynomials} $(S_{n}( x))_{n\geq 0}$ (see \cite{Wildberger}) play a central role in Wildberger's \emph{Rational Trigonometry}, an approach to plane geometry free of transcendental expressions. They encode rotations in the plane and are defined by the following recursion:
\begin{align*}
S_{0}( x) &=0,\ S_{1}( x) =x,\\
S_{n}( x) &=2( 1-2x) S_{n-1}( x) -S_{n-2}( x) +2x.
\end{align*}
\begin{proposition} \label{prop:ST} For every integer $n\geq 0$ we have $S_{n}( x) =\frac{1-T_{n}( 1-2x)}{2}$.
\end{proposition}
\begin{proof}
We define 
$$ S_{n}( x) :=\frac{1-T_{n}( 1-2x)}{2} \ \ \ \ \ ( *)_{n}.$$
Obviously, $ S_{0}( x) =0,\ S_{1}( x) =x$.
Assuming $(*)_{n}$ we want to deduce the recurrence for $S_{n}( x)$ from the recurrence for $T_{n}( x)$. We observe that $ ( *)_{n}$ is equivalent to 
$$ T_{n}( 1-2x) =1-2S_{n}( x) \ \ \ \ \ ( **)_{n}.$$ 
We have
$$T_{n}( 1-2x) =2( 1-2x) T_{n-1}( 1-2x) -T_{n-2}( 1-2x)$$
and hence
$$ S_{n}( x) =\frac{1-2( 1-2x) T_{n-1}( 1-2x) +T_{n-2}( 1-2x)}{2}.$$
From $ ( **)_{n-1}$ and $ ( **)_{n-2}$ we now obtain
\begin{align*}
S_{n}( x) &=\frac{1}{2}( 1-2( 1-2x)( 1-2S_{n-1}( x)) +1-2S_{n-2}( x))\\
&=\frac{1}{2}( 4x+2( 1-2x) 2S_{n-1}( x) -2S_{n-2}( x)) =2x+2( 1-2x) S_{n-1}( x)-S_{n-2}(x).
\end{align*}
\end{proof}
\begin{corollary}[M. Hirschhorn]\label{cor:Hirschhorn} The bivariate generating function of the spread polynomials is given by 
    \begin{align}
         S( x,t) =\sum _{n\geq 0} S_{n}( x) t^{n} =\frac{tx( 1+t)}{( 1-t)\left( 1-2t+t^{2} +4tx\right)}.
    \end{align}
\end{corollary}
\begin{proof}
Using formula \eqref{eqn:ChebGFT} we deduce
\begin{align*}
\sum _{n\geq 0} S_{n}( x) t^{n} &=\sum _{n\geq 0}\frac{1-T_{n}( 1-2x)}{2} t^{n} =\frac{1}{2}\sum _{n\geq 0} t^{n} -\frac{1}{2}\sum _{n\geq 0} T_{n}( 1-2x) t^{n}\\
&=\frac{1}{2}\sum _{n\geq 0} t^{n} -\frac{1}{2}\frac{1-t( 1-2x)}{1-2t( 1-2x) +t^{2}} =\frac{1}{2}\left(\frac{1}{1-t} -\frac{1-t( 1-2x)}{1-2t( 1-2x) +t^{2}}\right) \ \\
&=\frac{1}{2} \ \frac{1-2t( 1-2x) +t^{2} -( 1-t( 1-2x))( 1-t)}{( 1-t)( 1-2t( 1-2x) +t^{2}} \\
&=\frac{1}{2} \ \frac{1-2t+4tx+t^{2} -\left( 1-t-t+t^{2} +2tx-2t^{2} x\right)}{( 1-t)\left( 1-2t( 1-2x) +t^{2}\right)}\\
&=\frac{1}{2} \ \frac{2tx+2t^{2} x}{( 1-t)\left( 1-2t( 1-2x) +t^{2}\right)} =\frac{tx( 1+t)}{( 1-t)\left( 1-2t+t^{2} +4tx\right)}.
\end{align*}   
\end{proof}

\begin{corollary}\label{cor:wild} For each integer $n\geq 0$ we have $S_{n}(\sin^2(\theta))=\sin^2(n\theta)$.
\end{corollary}
\begin{proof}
From the Addition Theorem $\cos( 2z) =1-2\sin^{2}( z)$ and Proposition \ref{prop:ST} we deduce
$$S_{n}\left(\sin^{2}( \theta )\right) =\frac{1-T_{n}\left( 1-2\sin^{2}( \theta )\right)}{2} =\frac{1-T_{n}(\cos( 2\theta ))}{2} =\frac{1-\cos( 2n\theta )}{2} =\sin^{2}( n\theta ) .$$
\end{proof}
The \emph{spread matrix} $S=(S_{mn})_{m,n\geq 1}$ is then defined such that 
\begin{align*}
\begin{bmatrix}
S_{1}( x) & S_{2}( x) & S_{3}( x) & \dotsc 
\end{bmatrix} =\begin{bmatrix}
x & x^{2} & x^{3} & \dotsc 
\end{bmatrix} S.
\end{align*}
\begin{corollary}\label{cor:sqsinred} For each integer $n\geq 1$ we have
$ 2^{2n-2}\sin^{2n} (\theta )=\sum _{k=1}^{n} (-1)^{k-1}\binom{2n}{n-k}\sin^{2} (k\theta )$.
\end{corollary}
\begin{proof}
We use Theorem \ref{thm:powerred}(3) to deduce from Proposition \ref{prop:ST}
\begin{align*}
2^{2n-1}\sin^{2n}( \theta ) &=\frac{1}{2}\binom{2n}{n} +\sum _{k=1}^{n}( -1)^{k}\binom{2n}{n-k}\left( 1-2\sin^{2}( k\theta )\right)\\
&=\frac{1}{2}\binom{2n}{n} +\sum _{k=1}^{n}( -1)^{k}\binom{2n}{n-k} -2\sum _{k=1}^{n}( -1)^{k}\binom{2n}{n-k}\sin^{2}( k\theta )) .
\end{align*}
The first two terms cancel each other. In fact, we have
\begin{align*}
\frac{1}{2}\binom{2n}{n} +\sum _{k=1}^{n}( -1)^{k}\binom{2n}{n-k} &=\frac{1}{2}\binom{2n}{n} +\frac{1}{2}\sum _{j=1}^{n}( -1)^{n-j}\binom{2n}{j} +\frac{1}{2}\sum _{j=2n+1}^{2n}( -1)^{n-j}\binom{2n}{2n-j}\\
&=\frac{1}{2}( -1)^{n}\left(( -1)^{n}\binom{2n}{n} +\sum _{j=1}^{n}( -1)^{j}\binom{2n}{j} +\sum _{j=2n+1}^{2n}( -1)^{2n-j}\binom{2n}{2n-j}\right) \\
&=\frac{1}{2}( -1)^{n}\sum _{j=1}^{2n}( -1)^{j}\binom{2n}{j}
=\frac{1}{2}( -1)^{n}( 1-1)^{2n} =0,
\end{align*}
since $n\geq 1$.
\end{proof}

Let $\operatorname{Wild}$ be the subspace of the space $\operatorname{Trig}^0$ consisting of functions that vanish at $\theta =0$. 
Put $\shuffle ( \theta ) :=4\sin^{2} (\theta )$ and consider the two alternative bases for $\operatorname{Wild}$:
\begin{align}
\begin{bmatrix}
\shuffle ( \theta ) & \shuffle (2\theta ) & \shuffle (3\theta ) & \shuffle (4\theta ) & \shuffle (5\theta ) & \dotsc 
\end{bmatrix} ,\\
\begin{bmatrix}
\shuffle ( \theta ) & \shuffle ^{2} (\theta ) & \shuffle ^{3} (\theta ) & \shuffle ^{4} (\theta ) & \shuffle ^{5} (\theta ) & \dotsc 
\end{bmatrix}.
\end{align}
The idea is now to get rid of the powers of $4$ in the coefficients of $S_n$ by putting $Z_{n}( x) :=4S_{n}\left(\frac{x}{4}\right)$. We refer to the polynomials $Z_{n}( x)$
as the \emph{zpread polynomials}. It follows that 
\begin{align}
    Z_{n}( \shuffle ( \theta )) =\shuffle ( n\theta ),
\end{align}
and we will see in a moment that $Z_{n}( x) \in\Z[x]$.
We introduce the infinite \emph{zpread matrix} $(Z_{mn})_{m,n\geq 1}$ such that
\begin{align*}
    \begin{bmatrix}
        Z_1(x)& Z_2(x) & Z_3(x)&\dots
    \end{bmatrix}
    =\begin{bmatrix}
         x & x^2&x^3&\dots
    \end{bmatrix}Z.
\end{align*}
It turns out $Z$ is a matrix of even pyramidal numbers of dimension $\geq 2$.
\begin{theorem}\label{thm:wild}
For $m,n\geq 1$ we have $Z_{mn} =( -1)^{m+1} p_{n-m}^{[ 2m]}$.
\end{theorem}
\begin{proof}
Let $k\geq 0$ be an integer. From the generating function of the pyramidal numbers we deduce
$$\sum _{n\geq k} p_{n-k}^{[ 2m]} t^{n} =\sum _{n\geq 0} p_{n}^{[ 2m]} t^{n+k} =t^{k}\frac{1+t}{( 1-t)^{2m+1}} ,$$
and hence obtain 
$$\sum _{n\geq m}( -1)^{m+1} p_{n-m}^{[ 2m]} t^{n} =t^{m}\frac{1+t}{( 1-t)^{2m+1}}.$$
For the bivariate generating function this leads us to
\begin{align}
\label{eq:Zbivariate}
\sum _{n,m\geq 1}( -1)^{m+1} p_{n-m}^{[ 2m]} x^{m} t^{n} =-\sum _{m\geq 1}\frac{( -x)^{m} t^{m}}{( 1-t)^{2m}}\frac{1+t}{1-t} =\frac{1+t}{1-t}\frac{\frac{xt}{( 1-t)^{2}}}{1+\frac{xt}{( 1-t)^{2}}} =\frac{1+t}{1-t}\frac{xt}{( 1-t)^{2} +xt},
\end{align}
where we assume $-1<x,t<1$ and by default $p_{l}^{[ 2m]}=0$ for $l<0$.
The last expression in Equation \eqref{eq:Zbivariate} coincides with $4S\left(\frac{x}{4} ,t\right)$.    
\end{proof}

The mutual inverse linear systems corresponding to Corollary \ref{cor:sqsinred} turn out to be
\begin{multline}
\begin{bmatrix}
\shuffle ( \theta ) & \shuffle ^{2} (\theta ) & \shuffle ^{3} (\theta ) & \shuffle ^{4} (\theta ) & \shuffle ^{5} (\theta ) & \dotsc 
\end{bmatrix}\\
=\begin{bmatrix}
\shuffle ( \theta ) & \shuffle (2\theta ) & \shuffle (3\theta ) & \shuffle (4\theta ) & \shuffle (5\theta ) & \dotsc 
\end{bmatrix}
\begin{bmatrix}
1 & 4 & 15 & 56 & 210 &\cdots \\
 & -1 & -6 & -28 & -120 & \\
 &  & 1 & 8 & 45 & \\
 &  &  & -1 & -10 & \\
 &  &  &  & 1 & \\
\cdots &  &  &  &  & \cdots 
\end{bmatrix} ,
\end{multline}
\begin{multline}
\label{eq:sha}
\begin{bmatrix}
\shuffle ( \theta ) & \shuffle (2\theta ) & \shuffle (3\theta ) & \shuffle (4\theta ) & \shuffle (5\theta ) & \dotsc 
\end{bmatrix}\\
=\begin{bmatrix}
\shuffle ( \theta ) & \shuffle ^{2} (\theta ) & \shuffle ^{3} (\theta ) & \shuffle ^{4} (\theta ) & \shuffle ^{5} (\theta ) & \dotsc 
\end{bmatrix}\begin{bmatrix}
1 & 4 & 9 & 16 & 25&\cdots\\
 & -1 & -6 & -20 & -50&\\
 &  & 1 & 8 & 35&\\
 &  &  & -1 & -10&\\
 &  &  &  & 1 &\\
 \cdots &  &  &  &  &\cdots\\
\end{bmatrix},
\end{multline}
the latter matrix being $Z$, whose entries are relatively small. The former equation can be rephrased as 
\begin{align}
\shuffle (\theta )^n=\sum _{k=1}^{n}(-1)^{k-1}\binom{2n}{n-k} \shuffle\!(k\theta ).
\end{align}

The zpread matrix $Z$ can also be understood as the transpose of a Riordan array.
\begin{theorem} The transpose of the zpread matrix $Z$ is the Riordan array $Z^{T}=\left(\frac{1+x}{( 1-x)^{3}} ,\frac{-x}{( 1-x)^{2}}\right)$ with inverse
$(Z^{T})^{-1} =\left( BC^{2} ,-xC^{2}\right) =\left( C',-xC^{2}\right)$.
\end{theorem}

\begin{proof}
From Lemma \ref{lem:RAmult} and Theorem \ref{thm:RABC} item (2) and (3) we deduce
$$\left(\frac{1-x}{( 1+x)^{3}} ,\frac{x}{( 1+x)^{2}}\right) *\left( BC^{2} ,xC^{2}\right) =( 1,x).$$

To adjust the sign in front of $x$ we argue as follows.
Let $\rho \in x\mathbb{R} \llbracket x\rrbracket ^{\times }$ and consider the formal pullback \ $\rho ^{*} :\mathbb{R} \llbracket x\rrbracket \rightarrow \mathbb{R} \llbracket x\rrbracket ,\ f\mapsto f\circ \rho $, which is an automorphism of the algebra $(\mathbb{R} \llbracket x\rrbracket ,\cdot )$. Let us denote the product of two Riordan arrays $(g_{1} ,f_{1})$
and $( g_{2} ,f_{2})$ as $( g_{1} ,f_{1}) *( g_{2} ,f_{2}) =( g_{1}( g_{2} \circ f_{1}) ,f_{2} \circ f_{1}) =:( G,F)$. It then follows that
$$\left( \rho ^{*} g_{1} ,\rho ^{*} f_{1}\right) *( g_{2} ,f_{2}) =(( g_{1}( g_{2} \circ f_{1})) \circ \rho ,( f_{2} \circ f_{1}) \circ \rho ) =\left( \rho ^{*} G,\rho ^{*} F\right).$$
In the special case, when $\rho ( x) =-x$ we deduce
$$\left(\frac{1+x}{( 1-x)^{3}} ,\frac{-x}{( 1-x)^{2}}\right) *\left( BC^{2} ,xC^{2}\right) =( 1,-x).$$
But $( 1,-x)$ is an idempotent, i.e., $( 1,-x)^{*2} =( 1,x)$. Hence we obtain
$$( 1,x) =\left(\frac{1+x}{( 1-x)^{3}} ,\frac{-x}{( 1-x)^{2}}\right) *\left( BC^{2} ,xC^{2}\right) *( 1,-x) =\left(\frac{1+x}{( 1-x)^{3}} ,\frac{-x}{( 1-x)^{2}}\right) *\left( BC^{2} ,-xC^{2}\right) .$$
The verification of the well-known fact $BC^{2} =C'$ we leave to the reader.
\end{proof}

We refer to the following formula as the \emph{spreadometric series} since it specializes to $(1-x^2)^{-1}$ when $\theta=\pi/2$.
\begin{corollary} If $|x|<1$ then 
   $ \frac{1+x}{1-x}\frac{\sin^{2}( \theta )}{( 1-x)^{2} +4x\sin^{2}( \theta )} =\sum _{n\geq 0}\sin^{2}(( n+1) \theta ) x^{n}.$
\end{corollary}

\begin{proof}
We use the language of Riordan arrays. Let us put
$ \shuffle _{\theta }( x) :=\sum _{n\geq 0} \shuffle ( \theta )( \shuffle ( \theta ) x)^{n} =\frac{\shuffle ( \theta )}{1-\shuffle ( \theta ) x} .$
The right hand side of Equation \eqref{eq:sha} can be understood as
$$ Z^{T}\overrightarrow{\shuffle _{\theta }} =\left(\frac{1+x}{( 1-x)^{3}} ,-\frac{x}{( 1-x)^{2}}\right)\overrightarrow{\shuffle _{\theta }} =\overrightarrow{\frac{1+x}{( 1-x)^{3}}\frac{\shuffle ( \theta )}{1+\shuffle ( \theta )\frac{x}{( 1-x)^{2}}}} =\overrightarrow{\frac{1+x}{1-x}\frac{\shuffle ( \theta )}{( 1-x)^{2} +x\shuffle ( \theta )}}.$$
Hence Equation \eqref{eq:sha} means
$$ \frac{1+x}{1-x}\frac{\shuffle ( \theta )}{( 1-x)^{2} +x\shuffle \!( \theta )} =\sum _{n\geq 0} \shuffle (( n+1) \theta ) x^{n} .$$
The claim follows by writing this in terms of $ \sin^{2}( \theta )$.
\end{proof}

To memorize $S_{n}( x)$ one can proceed as the follows. Start with the square numbers and put below it the sequence whose differences are made by the square numbers. Underneath that sequence put 
the sequence whose differences are made by the latter. Continue recursively.
\begin{align}\label{eq:pyrarray}
 \begin{array}{ c c c c c c c c c c c c c c}
1 &  & 4 &  & 9 &  & 16 &  & 25 &  & 36 &  & 49& \cdots\\
 & 1 &  & 5 &  & 14 &  & 30 &  & 55 &  & 91 & &\\
 &  & 1 &  & 6 &  & 20 &  & 50 &  & 105 &  & 196&\\
 &  &  & 1 &  & 7 &  & 27 &  & 77 &  & 182 & &\\
 &  &  &  & 1 &  & 8 &  & 35 &  & 112 &  & 294&\\
 &  &  &  &  & 1 &  & 9 &  & 44 &  & 156 & &\\
 &  &  &  &  &  & 1 &  & 10 &  & 54 &  & 210&\\
 &  &  &  &  &  &  & 1 &  & 11 &  & 63 & &\\
 &  &  &  &  &  &  &  & 1 &  & 12 &  & 77&\\
 &  &  &  &  &  &  &  &  & 1 &  & 13 & &\\
 &  &  &  &  &  &  &  &  &  & 1 &  & 14&\\
 &  &  &  &  &  &  &  &  &  &  & 1 & &\\
 &  &  &  &  &  &  &  &  &  &  &  & 1& \cdots
\end{array}
\end{align}
In the resulting pattern take every second row and arrange them into a triangle.
The spread matrix $ S$ is then made by multiplying the $m$th row with $(-4)^{m-1}$, keeping in mind that the first row carries the row index $ m=1$:
\begin{align*}    
\begin{array}{ c c c c c c c c }
1 & 4 & 9 & 16 & 25 & 36 & 49 & \cdots\\
 & 1 & 6 & 20 & 50 & 105 & 196 & \\
 &  & 1 & 8 & 35 & 112 & 294 & \\
 &  &  & 1 & 10 & 54 & 210 & \\
 &  &  &  & 1 & 12 & 77 & \\
 &  &  &  &  & 1 & 14 & \\
 &  &  &  &  &  & 1 & \\ 
  &  &  &  &  &  &  & \cdots 
\end{array}
\mapsto
S=\left[\begin{array}{ c c c c c c c c }
1 & 4 & 9 & 16 & 25 & 36 & 49 & \cdots\\
 & -4 & -24 & -80 & -120 & -410 & -784 & \\
 &  & 16 & 128 & 560 & 1792 & 4704 & \\
 &  &  & -64 & -640 & -3456 & -13440 & \\
 &  &  &  & 256 & 3072 & 19712& \\
 &  &  &  &  & -1024 & -14336 & \\
 &  &  &  &  &  & 4096 & \\
 \cdots&  &  &  &  &  &  & \cdots
\end{array}\right].
\end{align*}
To obtain $Z$ one has to multiply the $m$th row with $(-1)^{m-1}$ instead.
The  transpose of the spread matrix is the Riordan array $S^{T}=\left(\frac{1+x}{( 1-x)^{3}} ,\frac{-4x}{( 1-x)^{2}}\right)$, whose inverse $\left( BC^{2} ,-xC^{2}/4\right)$ does of course not have integer entries.

\section{A conjecture of Goh and Wildberger}
\label{sec:GW}
Goh and Wildberger \cite{Butterfly,WildEgg} formulated an interesting conjecture concerning the factorizations of the spread polynomials over the ring $\mathbb{Z}[ x]$. We take here the liberty to formulate it in terms of the zpread polynomials and speculate a bit about numbers appearing in the coefficients of the factors. Using the zpread polynomials those become less obscure. As spread polynomials are encoding rotations in the plane\footnote{That is, they come straight from the almighty.} we find these unexpected phenomena intriguing. 

\begin{conjecture}
There are $\Phi _{d}( x) \in \mathbb{Z}[ x]$, $d\geq 1$,  with $\deg \Phi _{d}( x) =\phi ( d)$ such that for each $n\geq 1$ we have $Z_{n}( x) =\prod _{d|n} \Phi _{d}( x)$. For $d\geq 3$ the polynomial $\Phi _{d}( x) =\psi _{d}( x)^{2}$ is the square of an irreducible polynomial $\psi _{d}( x) =\mathbb{Z}[ x]$ with constant term $\psi _{d}( 0)  >0$. This constant term evaluates to $\psi _{d}( 0) =\prod _{d|n}\left(\frac{n}{d}\right)^{\mu ( d)}$ (see entry A014963 in \cite{OEIS}.) Here $\phi ( d)$ denotes the Euler totient function and $\mu ( d)$ the Möbius $\mu $-function. If $p\geq 5$ is a prime then $\psi _{p}( 1) =( -1)^{\phi ( p) /2}$.
\end{conjecture}
In order to convince the reader we provide some empirical evidence:
\begin{align*}
&\begin{array}{ c|c c c c c c c c }
d & 1 & 2 & 3 & 4 & 5 & 6 & 7 & 8\\
\hline
\Phi _{d}( x) & x & 4-x & ( 3-x)^{2} & ( 2-x)^{2} & \left( 5-5x+x^{2}\right)^{2} & ( 1-x)^{2} & \left( 7-14x+7x^{2} -x^{3}\right)^{2} & \left( 2-4x+x^{2}\right)^{2}
\end{array}\\
&\begin{array}{ c c c c }
9 & 10 & 11 & 12\\
\hline
\left( 3-9x+6x^{2} -x^{3}\right)^{2} & \ \left( 1-3x+x^{2}\right)^{2} & \left( 11-55x+77x^{2} -44x^{3} +11x^{4} -x^{5}\right)^{2} & \left( 1-4x+x^{2}\right)^{2}
\end{array}\\
&\begin{array}{ c c c }
13 & 14 & 15\\
\hline
\ \left( 13-91x+182x^{2} -156x^{3} +65x^{4} -13x^{5} +x^{6}\right)^{2} & \left( 1-6x+5x^{2} -x^{3}\right)^{2} & \left( 1-8x+14x^{2} -7x^{3} +x^{4}\right)^{2}
\end{array}\\
&
\begin{array}{ c c }
16 & 17\\
\hline
\left( 2-16x+20x^{3} -8x^{3} +x^{4}\right)^{2} & \left(17-204x+714x^{2} -1122x^{3} +935x^{4} -442x^{5} +119x-17x^{7} +x^{8}\right)^{2}
\end{array}.
\end{align*}
It comes as a surprise that the coefficients of $\psi _{d}( x)$ seem to almost coincide with certain columns of the array \eqref{eq:pyrarray}. Up to shifts, there are merely differences in the linear and/or constant coefficients.
The column index in the array relevant for $\psi _{d}( x)$ seems to be the coefficient in front of $-x^{\phi(d)/2-1}$.
From the empirical data we guess it is $\phi(d)-\mu(d)$ (see entry A053139 in \cite{OEIS}). We would like to mention that it seems to be the case that for each prime $ p>2$ we have $ \Phi _{p}( x) =\Phi _{2p}( 2-x)$. Furthermore, the evaluations of the $ \Phi _{n}( x)$ at the integers $x=0,1,2,3,4$ look interesting. One may ask which of those values $\Phi _{n}( x)$ appear for infinitely many $n$.

Clearly, for each $n\geq1$ we have that $Z_n$ maps the square $[0,4]^2$ into itself and all the zeros of $Z_n(x)$ have to be in $[0,4]$. The question for what $x\in[0,4]$ there are distinct $n,m$ such that $Z_n(x)=Z_m(x)$ deserves further attention. With regards to this let us mention some empirical observations.

\begin{conjecture}
Let $ n\geq 1$ be an integer and $ \varphi =\left( 1+\sqrt{5}\right) /2$ be the golden ratio. Then
    \begin{enumerate}
        \item  $ 2=Z_{n}( 2)$ \ if and only if $ n\equiv 1\bmod 2$,
  \item  $ 3=Z_{n}( 3)$ \ if and only if $ n\equiv 1,2\bmod 3$,
  \item $ 2+\varphi =Z_{n}( 2+\varphi )$ if and only if $ n\equiv 1,4\bmod 5$,
  \item  $ 3-\varphi =Z_{n}( 2+\varphi )$ if and only if $ n\equiv 2,3\bmod 5$.
    \end{enumerate}
\end{conjecture}
Johann Cigler \cite{Cigler} has noted that
\begin{align*}
    S_{2n}\!\left(x^{2}\right) =\left( 1-x^{2}\right) U_{2n-1}( x)^{2} ,\ \ S_{2n+1}\!\left(x^{2}\right) =T_{2n+1}( x)^{2},
\end{align*}
so that
\begin{align*}
   Z_{2n}\!\left( x^{2}\right) &=4S_{2n}\left(( x/2)^{2}\right) =\left( 4-x^{2}\right) U_{2n-1}( x/2)^{2},\\
Z_{2n+1}\!\left( x^{2}\right) &=4S_{2n+1}\left(( x/2)^{2}\right) =4T_{2n+1}( x/2)^{2}.
\end{align*}
Hence all the claims above correspond to statements about the Chebyshev polynomials, and those have been addressed in the literature. For example, the zeros of the Chebyshev polynomials correspond obviously the zeros of
$\cos(n\theta)$ and $\sin(n\theta)$, respectively. Hence the main task is to understand the combinatorics of how the linear factors are distributed to form the $\Phi_d$.

\bibliographystyle{amsalpha}
\bibliography{numerology}

\end{document}